\documentclass[3p]{elsarticle}

\usepackage{hyperref}

\journal{Approximate Reasoning}

\usepackage{amssymb} 
\usepackage{amsfonts}

\newtheorem{theorem}{Theorem}[section]
\newtheorem{proposition}[theorem]{Proposition}
\newtheorem{lemma}[theorem]{Lemma}

\newdefinition{definition}[theorem]{Definition}
\newdefinition{example}[theorem]{Example}
\newdefinition{remark}[theorem]{Remark}
\newdefinition{question}[theorem]{Question}

\newproof{proof}{Proof}

\newcommand{\UP}{\blacktriangle}                              
\newcommand{\DOWN}{\blacktriangledown}                        

\begin{document}
\begin{frontmatter}

\title{Defining rough sets as core--support pairs of three-valued functions}

\author{Jouni J{\"a}rvinen\corref{jj}}
\ead{jjarvine@utu.fi}
\address{Department of Mathematics and Statistics, University of Turku, 20014 Turku, Finland} 

\author{S{\'a}ndor Radeleczki}
\ead{matradi@uni-miskolc.hu}
\address{Institute of Mathematics, University of Miskolc, 3515~Miskolc-Egyetemv{\'a}ros, Hungary}

\cortext[jj]{Corresponding author}

\begin{abstract}
  We answer the question what properties a collection $\mathcal{F}$ of three-valued functions on a set $U$ must fulfill so that there
  exists a quasiorder $\leq$ on $U$ such that the rough sets determined by $\leq$ coincide with the core--support pairs
  of the functions in $\mathcal{F}$. Applying this characterization, we give a new representation of rough sets determined by equivalences in terms of
  three-valued {\L}ukasiewicz algebras of three-valued functions.
\end{abstract}


\end{frontmatter}


\section{Introduction} \label{Sec:Introduction}

Rough set defined by Z.~Pawlak \cite{Pawl82} are closely related to three-valued functions. In rough set theory, knowledge about objects of a universe of discourse
$U$ is given by an equivalence $E$ on $U$ interpreted so that $x \, E \, y$ if the elements $x$ and $y$ cannot be distinguished in terms of the information
represented by $E$. Each set $X \subseteq U$ is approximated by two sets: the lower approximation $X^\DOWN$ consists of elements which certainly belong to
$X$ in view of knowledge $E$, and the upper approximation $X^\UP$ consists of objects which possibly are in $X$. Let $\mathbf{3} = \{0,u,1\}$ be the 3-element
set in which the elements are ordered by $0 < u < 1$. For any $X \subseteq U$, we can define a three-valued function $f$ such that $f(x) = 0$ if $x$ does not
belong to $X^\UP$, that is, $x$ is interpreted to be certainly outside $X$. We set $f(x) = 1$ when $x \in X^\DOWN$, meaning that $x$ certainly belongs to $X$.
If $x$ belongs to the set-difference $X^\UP \setminus X^\DOWN$, which is the actual area of uncertainty, we set $f(x) = u$.

On the other hand, in fuzzy set theory the `support' of a fuzzy set is a set that contains elements with degree of membership greater than $0$ and the
`core' is a set containing elements with degree of membership equal to $1$. Naturally, each $3$-valued function can be considered as
a fuzzy set, and for $f \colon U \to \mathbf{3}$, its core $C(f)$ can be viewed as a subset of $U$ consisting of elements which
certainly belong to the concept represented by $f$, and the support $S(f)$ may be seen as a set of objects possible belonging to the concept represented by $f$.
Obviously, $C(f) \subseteq S(f)$ for any three-valued function $f$. Note also that different roles of three-valued information, such as vague, incomplete or
conflicting information are considered in \cite{Ciucci2014}.

We call pairs $(A,B)$ of subsets of $U$ such that $A \subseteq B$ as `approximation pairs'. The motivation for this name is that
$X^\DOWN \subseteq X^\UP$ for all $X \subseteq U$ if and only if the relation defining the approximations is serial  (see \cite{Jarv07}, for instance). 
A relation $R$ on $U$ is serial if each element of $U$ is $R$-related to at least one element. Therefore, for a serial relation $R$ on $U$ and a subset
$X \subseteq U$, the pair $(X^\DOWN, X^\UP)$, called `rough set', is an approximation pair, and thus an approximation pair can be seen as a generalization of
a rough set $(X^\DOWN, X^\UP)$. 

The concept of `three-way decisions' is recently introduced to relate with the three regions determined by rough approximations; see \cite{Yao2012}, for example.
The idea of three-way decisions is to divide, based on some criteria, the universe $U$ into three pairwise disjoint  regions, called the
positive (\textsc{pos}),  negative (\textsc{neg}), and boundary (\textsc{bnd}) regions, respectively. Obviously, each such a three-way decision forms
an approximation pair $(\textsc{pos}, \textsc{pos} \cup \textsc{bnd})$ and each approximation pair $(A,B)$ induces a three-way decision, where the
positive area is $A$, the boundary equals the difference $B \setminus A$, and $B^c$ forms the negative region. More generally, there is a
one-to-one correspondence between approximation pairs and 3-valued functions. 

The three-valued chain $\mathbf{3}$ can be equipped with various algebraic structures. It is known that the 3-valued chain $\mathbf{3}$ forms a complete completely
distributive lattice, a Heyting algebra, a 3-valued {\L}ukasiewicz algebra, a semi-simple Nelson algebra, and a regular double Stone algebra, for instance.
The operations of these algebras are uniquely defined in $\mathbf{3}$ and they can be extended pointwise to $\mathbf{3}^U$, the set of all three-valued functions
on $U$. The set $\mathbf{3}^U$ is canonically ordered pointwise and the set $\mathcal{A}(U)$ of the approximation pairs of $U$ is ordered naturally by the
coordinatewise $\subseteq$-relations. The sets $\mathbf{3}^U$ and $\mathcal{A}(U)$ form isomorphic ordered structures (see Section~\ref{Sec:Three-valued}).
In addition, this means that $\mathcal{A}(U)$ has the above-mentioned algebraic structures 'lifted' from $\mathbf{3}$ (via $\mathbf{3}^U$).

We denote by $\mathcal{RS}$ the set of all rough sets $\{ (X^\DOWN, X^\UP) \mid X \subseteq U\}$ defined by some binary relation $R$ on $U$. If $R$ is an
equivalence, $\mathcal{RS}$ is a Stone algebra \cite{PomPom88}. In \cite{Com93} this result was improved by showing that $\mathcal{RS}$
forms a regular double Stone algebra. The three-valued {\L}ukasiewicz algebras defined by $\mathcal{RS}$ were considered in
\cite{BanChak96, BanChak97, Iturrioz99, Pagliani97}. P.~Pagliani \cite{Pagliani97} showed how a semisimple Nelson algebra can be
defined on $\mathcal{RS}$. In addition, as Stone algebras, three-valued {\L}ukasiewicz algebras, or Nelson algebras, they are subalgebras of $\mathcal{A}(U)$.
If $R$ is a quasiorder, $\mathcal{RS}$ is a complete polarity sublattice of $\mathcal{A}(U)$ as noted in \cite{JRV09}. Because rough sets
are approximations pairs, there is a complete polarity sublattice $\mathcal{F}$ of $\mathbf{3}^U$ such that its approximation pairs $\mathcal{A}(\mathcal{F})$
equal $\mathcal{RS}$. 

In this work, we obtain sufficient and necessary conditions (C1)--(C3) under which $\mathcal{A(F)} = \mathcal{RS}$ holds, where $\mathcal{F}$ is a complete polarity sublattice 
of $\mathbf{3}^{U}$ and $\mathcal{RS}$ is induced by a quasiorder. 
In the special case of an equivalence $R$, we have $\mathcal{A(F)} = \mathcal{RS}$ exactly when $\mathcal{F}$ satisfies (C1)--(C3) and
forms a {\L}ukasiewicz subalgebra of $\mathbf{3}^U$. The latter result uses our
earlier result stating that $\mathcal{RS}$ forms a 3-valued {\L}ukasiewicz algebra just in case $R$ is an equivalence relation \cite{JarRad11}. We also
prove that such an $\mathcal{F}$ forms a 3-valued {\L}ukasiewicz algebra when it is closed with respect to any of
the operations $^*$, $^+$, ${\triangledown}$, ${\vartriangle}$, $\to$, $\Rightarrow$ defined in $\mathbf{3}^U$ (see Proposition~\ref{Prop:ImplyLuka}).

This paper is structured as follows. In the next section, we consider the set $\mathbf{3}^U$ of all 3-valued functions on $U$
and the approximation pairs $\mathcal{A}(U)$ defined by them. We point out that $\mathbf{3}^U$ and $\mathcal{A}(U)$ form isomorphic
complete lattices. Also the basic definitions and facts related to rough sets are recalled in this section.
In Section~\ref{Sec:algebras}, we note how $\mathbf{3}^U$ forms a Heyting algebra, a three-valued {\L}ukasiewicz algebra, a semisimple Nelson algebra,
and a regular double Stone algebra. The operations on all these algebras are defined pointwise from the operations of $\mathbf{3}$.
Because $\mathcal{A}(U)$ is isomorphic to $\mathbf{3}^U$, all the mentioned algebras can be defined on
$\mathcal{A}(U)$, too. We describe these operations on $\mathcal{A}(U)$ in detail. We end this section by noting that if a complete polarity sublattice
$\mathcal{F}$ of $\mathbf{3}^U$ is closed with respect to at least one of the operations $^*$, $^+$, ${\triangledown}$, ${\vartriangle}$, $\to$, $\Rightarrow$
defined in  $\mathbf{3}^U$, then $\mathcal{F}$ is closed with respect to all these operations.

It is well-known that there is a one-to-one correspondence between quasiorders and Alexandrov topologies. In Section~\ref{sec:Alex},
we consider Alexandrov topologies defined by complete sublattices of $\mathbf{3}^U$. For a quasiorder $\leq$,
a necessary condition for $\mathcal{A(F)} = \mathcal{RS}$ to hold is that the collections $C(\mathcal{F})$ and $S(\mathcal{F})$ of
the cores and the supports of the maps in $\mathcal{F}$, respectively, form dual Alexandrov topologies. Moreover $C(\mathcal{F})$ must equal
$\wp(U)^\DOWN$, the set of lower approximations of subsets of $U$, and $S(\mathcal{F})$ needs to coincide with $\wp(U)^\UP$,
the set of upper approximations.

Together with Pagliani the authors of the current work presented in \cite{JaPaRa13} a representation of quasiorder-based rough sets, stating that
\begin{equation} \label{eq:increasing}
\mathcal{RS} = \{(A, B) \in \wp(U)^\DOWN \times \wp(U)^\UP \mid A \subseteq B \quad \mbox{and} \quad \mathcal{S} \subseteq A \cup B^c \},
\end{equation}
where $\mathcal{S}$ is the set of such elements that they are $\leq$-related only to itself. This representation appears simple compared
to the representation presented here. The fact is that there is already a lot of structural information in each $(A,B)$-pair of (\ref{eq:increasing}),
because each such pair is defined by a single quasiorder $\leq$. But if we just pick an arbitrary collection $\mathcal{F}$ of
three-valued functions (or approximation pairs), nothing is connecting these functions together. Probably for this reason the conditions under
which a complete polarity sublattice $\mathcal{F}$ of $\mathbf{3}^U$ is such that $\mathcal{A(F)} = \mathcal{RS}$ need to be more complicated
than (\ref{eq:increasing}).
Studying these characteristic properties possibly reveals something new and essential about the nature of rough sets determined by a quasiorder or by an equivalence.

Some concluding remarks end the article.

\section{Three-valued functions and approximations} \label{Sec:Three-valued}

We consider three-valued functions $f \colon U \to \mathbf{3}$ defined on a universe $U$, where $\mathbf{3}$ stands for the three-elemented chain $0<u<1$.  
The set of such functions ${\mathbf 3}^U$ may be ordered \emph{pointwise} by using the order of $\mathbf{3}$:
\[
f \leq g \iff f(x) \leq g(x) \ \mbox{for all} \ x \in U.
\]
With respect to pointwise  order,  $\mathbf{3}^U$ forms a complete lattice such that
\[
\big (\bigvee \mathcal{H} \big ) (x) = \max \{ f(x) \mid f \in \mathcal{F}\}
\quad \mbox{ and } \quad
\big (\bigwedge \mathcal{H} \big ) (x) = \min \{ f(x) \mid f \in \mathcal{F}\}
\]
for any $\mathcal{H} \subseteq \mathbf{3}^U$. 
The map $\bot \colon x \mapsto 0$ is the least element and  $\top \colon x \mapsto 1$ is the greatest element of  $\mathbf{3}^U$. 

It is well-known that $\mathbf{3}$ is equipped with several operations such as Heyting implication $\Rightarrow$, polarity ${\sim}$, pseudocomplement $^*$,
dual pseudocomplement $^+$, possibility $\triangledown$ and necessity $\vartriangle$ of three-valued {\L}ukasiewicz algebras and Nelson implication $\to$.
Any $n$-ary, $n \geq 0$,  operation $\phi$ on $\mathbf{3}$ can be `lifted' pointwise to an operation $\Phi$ on the set $\mathbf{3}^U$ by defining for the maps
$f_1, \ldots, f_n \in \mathbf{3}^U$ a function $\Phi(f_1, \ldots, f_n)$ in  $\mathbf{3}^U$ by setting
\[ (\Phi(f_1,\ldots,f_n))(x) = \phi(f_1(x),\ldots,f_n(x)) \ \mbox{ for all $x \in U$.} \]
The operation $\Phi$ then satisfies the same identities in $3^U$ as $\phi$ satisfies in $\mathbf{3}$.

Rough sets are pairs consisting of a lower and an upper approximation of a set. In this work, a generalization of such pairs are in an essential role.
Let $A,B \subseteq U$. We say that $(A,B)$ is an \emph{approximation pair} if $A \subseteq B$. We denote by $\mathcal{A}(U)$ the set of all approximation
pairs on the set $U$. The set $\mathcal{A}(U)$ can be ordered \emph{componentwise} by setting
\[
(A,B)\leq(C,D) \iff A\subseteq C \ \mbox{and} \ B\subseteq D.
\]
for all $(A,B),(C,D) \in \mathcal{A}(U)$. With respect to the componentwise order, $\mathcal{A}(U)$ is a complete sublattice of 
$\wp(U) \times \wp(U)$, where $\wp(U)$ denotes the family of all subsets of $U$. If $\{ (A_i, B_i) \mid i \in I\} \subseteq \mathcal{A}(U)$, then
\[ 
\bigvee_{i \in I} (A_i, B_i ) = \big (\bigcup_{i \in I} A_i, \bigcup_{i \in I} B_i \Big )
\quad \mbox{ and } \quad 
\bigwedge_{i \in I} (A_i, B_i) = \Big (\bigcap_{i \in I} A_i, \bigcap_{i \in I} B_i \Big ). \]
Note that $\mathcal{A}(U)$ can be viewed as an instance of 
\[ B^{[2]} = \{ (a,b) \in B^2 \mid a \leq b\},\]
where $B$ is a Boolean lattice. It is well known that $B^{[2]}$ is a regular double Stone lattice \cite{Gratzer}.

Every $f \in \mathbf{3}^{U}$ is completely determined by two sets
\[
C(f) = \{x\in U\mid f(x)=1\} \quad \mbox{ and } \quad S(f) = \{x \in U \mid f(x) \geq u\}
\]
called the \emph{core} and the \emph{support} of $f$, respectively. Clearly, $C(f)\subseteq S(f)$, and the pair $(C(f), S(f))$ is called the \emph{approximation pair of} $f$. 
Note that if $f(x) \in \{0,1\}$ for all $x \in U$, then $C(f) = S(f)$.

\begin{proposition} \label{Prop:Correspondence}
The mapping
\[
\varphi \colon \mathbf{3}^{U}\to\mathcal{A}(U), \quad f \mapsto (C(f),S(f)) 
\]
is an order-isomorphism.
\end{proposition}

\begin{proof} We first show that $\varphi$ is an order-embedding, that is,
\[ f\leq g \iff (C(f),S(f))\leq(C(g),S(g)). \] 
Assume $f \leq g$, that is, $f(x) \leq g(x)$ for all $x \in U$. If $x \in C(f)$, then $g(x) \geq f(x) = 1$ and $x \in C(g)$. So, $C(f) \subseteq C(g)$.
Similarly, if $x \in S(f)$, then $g(x) \geq f(x) \geq u$ and $x \in S(g)$. Therefore, also $S(f) \subseteq S(g)$ and we have proved $ (C(f),S(f))\leq(C(g),S(g))$.

Conversely, assume $ (C(f),S(f))\leq(C(g),S(g))$. If $f(x) = 0$, then trivially $f(x) \leq g(x)$. If $f(x) = u$, then $x \in S(f) \subseteq S(g)$ and $g(x) \geq u = f(x)$.
If $f(x) = 1$, then $x \in C(f) \subseteq C(g)$ and $g(x) = f(x)$. Hence, $f(x) \leq g(x)$ for all $x \in U$, that is, $f \leq g$.

We need to show that $\varphi$ is a surjection. Suppose $(A,B) \in \mathcal{A}(U)$. Let us define a function $f_{(A,B)}$ by
\begin{equation} \label{Eq:characteristic}
f_{(A,B)}(x)=\left\{
\begin{array}{ll}%
1 & \mbox{if $x \in A$,}\\
u & \mbox{if $x \in B\setminus A$,}\\
0 & \mbox{if $x \notin B$.}
\end{array}
\right.
\end{equation}
Now 
\[ \varphi(f_{(A,B)}) = (C(f_{(A,B)}), S(f_{(A,B)})) = (A,A\cup(B\setminus A))=(A,B). \]
We have now proved that $\varphi$ is an order-isomorphism. \qed
\end{proof}

A complete lattice $L$ is \emph{completely distributive} if for any doubly indexed
subset $\{x_{i,\,j}\}_{i \in I, \, j \in J}$ of $L$, we have
\[
\bigwedge_{i \in I} \Big ( \bigvee_{j \in J} x_{i,\,j} \Big ) = 
\bigvee_{ f \colon I \to J} \Big ( \bigwedge_{i \in I} x_{i, \, f(i) } \Big ), \]
that is, any meet of joins may be converted into the join of all possible elements obtained by taking the meet over $i \in I$ of
elements $x_{i,\,k}$\/, where $k$ depends on $i$.

The power set lattice $\wp(U)$ is a well-known completely distributive lattice \cite{DaPr02}. In $\wp(U) \times \wp(U)$, the joins and meets are
formed coordinatewise, so $\wp(U) \times \wp(U)$ is a completely distributive lattice. Also a complete sublattice of a completely distributive lattice is
clearly completely distributive. Thus, $\mathcal{A}(U)$ and $\mathbf{3}^{U}$ are completely distributive.

\begin{lemma} \label{lem:ContinuousMeetsJoins}
If $\mathcal{F} \subseteq \mathbf{3}^{U}$, then
\begin{enumerate}[\rm (i)]
\item $C \big(  \bigvee \mathcal{F} \big)  = \bigcup \{ C(f) \mid f \in \mathcal{F} \}$ \ and \
  $S \big(  \bigvee \mathcal{F} \big)  = \bigcup \{ S(f) \mid f \in \mathcal{F} \}$;
\item $C \big(  \bigwedge \mathcal{F} \big)  = \bigcap \{ C(f) \mid f \in \mathcal{F} \}$ \ and \
  $S \big(  \bigwedge \mathcal{F} \big)  = \bigcap \{ S(f) \mid f \in \mathcal{F} \}$.
\end{enumerate}
\end{lemma}

\begin{proof} By Proposition~\ref{Prop:Correspondence}, the map $\varphi \colon f \to (C(f),S(f))$ is an order-isomorphism. Hence, it preserves all meets and
joins, and 
$\varphi (\bigvee \mathcal{F}) = \bigvee \{ \varphi(f) \mid f \in \mathcal{F} \}$.
By definition, $\varphi (\bigvee \mathcal{F}) = (C(\bigvee \mathcal{F}),S(\bigvee \mathcal{F}))$
and  $\bigvee \{ \varphi(f) \mid f \in \mathcal{F} \} = \bigvee \{ (C(f),S(f)) \mid f \in \mathcal{F} \}$. Because $\mathcal{A}(U)$ is a complete sublattice of
$\wp(U) \times \wp(U)$,  $\bigvee \{ (C(f),S(f)) \mid f \in \mathcal{F} \} = (\bigcup\{ C(f) \mid f \in \mathcal{F}\}, \bigcup\{ S(f) \mid f \in \mathcal{F}\})$.
Combining all these, we can write
\[
\Big (  C \big (  \bigvee \mathcal{F} \big ), S \big ( \bigvee \mathcal{F} \big ) \Big )
 = \varphi \big (  \bigvee \mathcal{F} \big)  = \bigvee_{f \in \mathcal{F}}\varphi(f) = \bigvee_{f \in \mathcal{F} } ( C(f), S(f) )  \\
 = \Big( \bigcup_{f \in \mathcal{F} }C(f), \bigcup_{f \in \mathcal{F} }S(f)\Big ),
\]
which proves (i), and (ii) is treated analogously. \qed
\end{proof}

The set of approximation pairs corresponding to a family $\mathcal{F} \subseteq \mathbf{3}^{U}$ is defined as
\[ \mathcal{A}(\mathcal{F}) = \{(C(f),S(f)) \mid f \in \mathcal{F\}}. \]
Obviously, for any $\mathcal{F} \subseteq \mathbf{3}^{U}$, the ordered sets $\mathcal{F}$ and $\mathcal{A}(\mathcal{F})$ are order-isomorphic,
whenever $\mathcal{F}$ is ordered pointwise and  $\mathcal{A}(\mathcal{F})$ coordinatewise.

Rough sets were introduced by Z. Pawlak \cite{Pawl82}. According to Pawlak's original definition, our knowledge
about objects $U$ is given by an equivalence relation. Equivalences are reflexive, symmetric and transitive binary relations.
An equivalence $E$ on $U$ is interpreted so that $x \, E \, y$ if the elements $x$ and $y$ cannot be distinguished by their known properties.
In the literature, numerous studies can be found on rough sets in which equivalences are replaced by different types of so-called
\emph{information relations} reflecting, for instance, similarity or preference between the elements of $U$
(see e.g.~\cite{DemOrl02, Orlowska1998}).

Let $U$ be a set and let $R$ be a binary relation on $U$. For any $x \in U$, we denote $R(x) = \{ y \mid x \, R \, y\}$. 
For all $X \subseteq U$, the \emph{lower} and \emph{upper approximations} of $X$ are defined by
\[ X^\DOWN = \{ x \in U \mid R(x) \subseteq X \} \mbox{ \quad and \quad }  X^\UP = \{ x \in U \mid R(x) \cap X \ne \emptyset \}, \]
respectively. The set $X^\DOWN$ may be interpreted as the set of elements that are \textit{certainly} in $X$, because  all elements 
to which $x$ is $R$-related are in $X$. Analogously, $X^\UP$ can be considered as the set of all elements that are \textit{possibly} in $X$, since in $X$ there
is at least one element to which $x$ is $R$-related. Quasiorders are reflexive and transitive binary relations.
For instance, a quasiorder $R$ can be interpreted as a \emph{specialization order}, where $x \, R \, y$ may be read as `$y$ is a specialization of $x$'.
In \cite{Ganter2008}, a specialization order is viewed as `non-symmetric indiscernibility' such that each element is indiscernible with all its
specializations, but not necessarily the other way round. Then, in our interpretation, $x \in X^\UP$ means that there is at least one specialization $y$ in $X$, which 
cannot be discerned from $x$. Similarly, $x$ belongs to $X^\DOWN$ if all its specializations are in $X$; this is then interpreted so that $x$ needs to be in $X$ in
the view of the knowledge $R$.

For all $X \subseteq U$, the pair $(X^\DOWN, X^\UP)$ is called the \emph{rough set of $X$}. The set of all rough sets is denoted by 
\[ \mathcal{RS} = \{ (X^\DOWN, X^\UP) \mid X \subseteq U \}. \]
Like any set of approximations,  $\mathcal{RS}$ is ordered coordinatewise:
\[
(X^\DOWN, X^\UP) \leq (Y^\DOWN, Y^\UP) \iff X^\DOWN \subseteq Y^\DOWN \quad \mbox{and} \quad X^\UP \subseteq Y^\UP.
\]
In this work, we consider relations $R$ which are at least reflexive. Then $X^\DOWN \subseteq X \subseteq X^\UP$, and therefore  each rough set  $(X^\DOWN,X^\UP)$ 
can be considered as an approximation pair in the above sense. 
For reflexive relations, $\mathcal{RS}$ is not necessarily a lattice. In fact, it is known that there are tolerances, that is,
reflexive and symmetric binary relations, such that  $\mathcal{RS}$ is not a lattice; see \cite{Jarv07}.

We restrict ourselves to the case in which $R$ is a quasiorder or an equivalence. This has the advantage that the rough sets algebras $\mathcal{RS}$ 
are complete polarity sublattices of $\mathcal{A}(U)$. This also means that there exists a complete polarity sublattice $\mathcal{F}$ of $\mathbf{3}^U$
such that its approximation pairs $\mathcal{A}(\mathcal{F})$ are equal to $\mathcal{RS}$.

In this work, we consider the question, what properties of complete polarity sublattice $\mathcal{F}$ of  $\mathbf{3}^{U}$ must additionally satisfy so that 
$\mathcal{A}(\mathcal{F}) = \mathcal{RS}$ holds, where $\mathcal{RS}$ is the collection of rough sets induced by a quasiorder or by an equivalence on $U$.

\section{Algebras defined on $\mathbf{3}^U$ and $\mathcal{A}(U)$} \label{Sec:algebras}

For an ordered set $(P,\leq)$, a mapping ${\sim} \colon P \to P$ satisfying
\[
{\sim}{\sim} x = x \quad \mbox{and} \quad x  \leq y \ \mbox{implies} \ {\sim} x \geq {\sim} y  
\]
is called a \emph{polarity}. Such a polarity $\sim$ is an order-isomorphism from $(P,\leq)$ to its dual $(P,\geq)$.
This means that $P$ is \emph{self-dual} to itself. Let us define an operation ${\sim}$ on $\wp(U) \times \wp(U)$ by
\[
{\sim}(A,B) = (B^{c},A^{c}),
\]
where for any $X \subseteq U$, $X^c$ denotes the \emph{complement} $U \setminus X$ of $X$.
We call the pair ${\sim}(A,B)$ as the \emph{opposite} of $(A,B)$.
Obviously, $\sim$ is a polarity. Let $L$ be a (complete) lattice with polarity. If $S$ is a (complete) sublattice of $L$
closed with respect to $\sim$, we say that $S$ is a \emph{(complete) polarity sublattice} of $L$.
Because $A \subseteq B$ implies $B^{c} \subseteq A^{c}$,  $\mathcal{A}(U)$ is a complete polarity sublattice of $\wp(U) \times \wp(U)$.

For any binary relation $R$ on $U$, the approximation operations $^\DOWN$ and $^\UP$ are dual, that is, for $X \subseteq U$,
\[ X^{c \UP} = X^{\DOWN c} \quad \mbox{and} \quad  X^{c \DOWN} = X^{\UP c}  .\]
This implies that for $(X^\DOWN, X^\UP) \in \mathcal{RS}$,
\[ {\sim}(X^\DOWN, X^\UP) = (X^{\UP c}, X^{\DOWN c}) =  (X^{c \DOWN}, X^{c \UP}). \]
Therefore, $\sim$ is a well-defined polarity also in $\mathcal{RS}$.

\begin{remark} \label{rem:orthopairs}
Our study has some resemblance to the study of so-called `orthopairs' by G.~Cattaneo and D.~Ciucci  \cite{Cattaneo2018}. 
They define \emph{De~Morgan posets} as bounded ordered sets with a polarity $\sim$. A pair $(x,y)$ is called an \emph{orthopair} if $x \leq {\sim}y$.
By introducing additional properties to a De~Morgan poset, one gets different algebraic structures of orthopairs.

Let $U$ be a set. Then $\wp(U)$ equipped with a set-theoretical complement $^c$ forms a De~Morgan poset. It is clear that
$(A,B) \in \mathcal{A}(U)$ if and only if $(A,B^c)$ is an orthopair. Orthopairs can be viewed as a generalization of
\emph{disjoint representation of rough sets} introduced by P.~Pagliani in \cite{Pagliani97}. 
\end{remark}

A \emph{De Morgan algebra} $(L,\vee,\wedge,\sim,0,1)$ is such that $(L,\vee,\wedge,0,1)$ is a bounded
distributive lattice and $\sim$ is a polarity. The operation $\sim$ can be defined also by the identities:
\[ {\sim}{\sim} x = x \quad \mbox{and} \quad {\sim} (x \wedge y) = {\sim} x \vee {\sim} y.\]
 
\begin{example} \label{exa:demorgan}
The chain $\mathbf{3}$ is a De~Morgan algebra in which $\sim$ is defined by:
\begin{center}
\begin{tabular}{c|c}
$x$ & ${\sim} x$\\\hline
$0$ & $1$\\
$u$ & $u$\\
$1$ & $0$
\end{tabular}
\end{center}
Also $(\mathbf{3}^{U},\vee,\wedge ,\sim,\bot,\top)$ is a De~Morgan algebra, where for any $f\in\mathbf{3}^{U}$, 
${\sim} f$ is defined pointwise by 
\[ ({\sim} f)(x) = {\sim} f(x). \] 
\end{example} 

\begin{lemma} \label{lem:support_demorgan}
If $f \in \mathbf{3}^U$, then
\[  C({\sim} f) = S(f)^{c} \qquad \mbox{and} \qquad S({\sim} f)=C(f)^{c}.\] 
\end{lemma}

\begin{proof}
For $x \in U$,
\[
x \in C({\sim}f)  \iff  ({\sim}f)(x)  = 1   \iff  {\sim} f(x) = 1  \iff  f(x) = 0  \iff  x \notin S(f)  \iff x \in S(f)^c,       
\]
which proves the first claim. Since ${\sim}{\sim}f = f$, we obtain
\[ S({\sim}f) = C({\sim}{\sim}f)^c =  C(f)^c.  \] \qed 
\end{proof}

Now $(\mathcal{A}(U),\vee,\wedge, {\sim},(\emptyset,\emptyset),(U,U))$ is a De~Morgan algebra isomorphic to 
$(\mathbf{3}^{U},\vee,\wedge,\sim,\bot,\top)$.
It is easy to see that $\varphi(\bot) = (\emptyset,\emptyset)$ and $\varphi(\top) = (U,U)$. By Proposition~\ref{Prop:Correspondence}, it is enough to show that
\[
\varphi({\sim} f) = (  C({\sim}f) , S({\sim}f))  = (S(f)^{c}, C(f)^{c})  = {\sim}(C(f), S(f)) = {\sim} \varphi(f).
\]

Following A.~Monteiro~\cite{Monteiro63}, we can define a \emph{three-valued {\L}ukasiewicz algebra} as an algebra $(L,\vee,\wedge,{\sim},{\triangledown},0,1)$
such that $(L,\vee,\wedge,\sim,0,1)$ is a De Morgan algebra and ${\triangledown}$ is an unary operation, called the \emph{possibility operator}, that satisfies the identities:
\begin{enumerate}[({L}1)]
\item ${\sim} x \vee {\triangledown} x=1$,
\item ${\sim x}  \wedge x = {\sim} x   \wedge {\triangledown} x$, and 
\item ${\triangledown}(x \wedge y) = {\triangledown}x \wedge {\triangledown} y$.
\end{enumerate}
Let us recall from \cite{Monteiro63} that the following facts hold for all $x \in L$,
\[ x \leq {\triangledown} x, \qquad {\triangledown} 0 = 0, \quad {\triangledown} 1 = 1, \quad {\triangledown}{\triangledown} x = {\triangledown} x, \quad 
{\triangledown}(x \vee y) =  {\triangledown}x \vee  {\triangledown}y.   \]
In addition $x \leq y$ implies ${\triangledown x} \leq  {\triangledown y}$. The \emph{necessity operator} is defined by 
\[ {\vartriangle} x = {\sim} {\triangledown} {\sim}  x. \]
The operation $\vartriangle$ can be seen as a dual operator of $\triangledown$, so $\vartriangle$ satisfies the dual assertions of the above. Also $\vartriangle$ and $\triangledown$ have
some mutual connections, for instance:
\[ {\vartriangle}\!{\triangledown} x =  {\triangledown} x \qquad \mbox{and} \qquad {\triangledown}\!{\vartriangle} x =  {\vartriangle} x .\]
{\L}ukasiewicz algebras satisfy the following \emph{determination principle} by Gr.~C.~Moisil (see e.g. \cite{Moisil63}):
\[  {\vartriangle}x =  {\vartriangle}y \quad \mbox{and} \quad {\triangledown}x = {\triangledown}y \quad \mbox{imply} \quad x = y. \]

It is known  \cite{Iturrioz99} that if $\mathcal{RS}$ is defined by an equivalence relation on $U$, then 
it forms a 3-valued {\L}ukasiewicz algebra such that
\[  {\vartriangle} (X^\DOWN,X^\UP) = (X^\DOWN, X^\DOWN) \qquad \mbox{and} \qquad {\triangledown} (X^\DOWN, X^\UP) = (X^\UP, X^\UP) .\]

\begin{example}\label{exa:3valued}
On the chain  $\mathbf{3}$ the operations ${\vartriangle}$ and ${\triangledown}$ are defined as in the following table:
\begin{center}
\begin{tabular}{c|cc}
$x$ & ${\vartriangle}x$ & ${\triangledown}x $\\ \hline
$0$ & $0$ & $0$ \\
$u$ & $0$ & $1$ \\
$1$ & $1$ & $1$
\end{tabular}
\end{center}
For a map $f \in \mathbf{3}^{U}$, the functions ${\vartriangle}f$ and ${\triangledown}f$ are defined pointwise, that is,
\[  ({\vartriangle}f)(x) =   {\vartriangle}f(x) \quad \mbox{and} \quad ({\triangledown}f)(x)  = {\triangledown}f(x) .\]

Also $\mathcal{A}(U)$ forms a three-valued {\L}ukasiewicz algebra in which 
\[ {\vartriangle}(A,B) = (A,A) \qquad \mbox{and} \qquad {\triangledown}(A,B) = (B,B).\] 
\end{example}

\begin{lemma} \label{lem:support_triangle}
If $f \in \mathbf{3}^U$, then
\[ C({\triangledown}f) =  S({\triangledown}f) = S(f). \]
\end{lemma}

\begin{proof} Let $x \in U$. Then,
\[
x \in C({\triangledown}f)  \iff ({\triangledown}f)(x) = 1 \iff {\triangledown}f(x) = 1 \iff f(x) \geq u \iff x \in S(f). 
\]
Because $({\triangledown}f)(x) \in \{0,1\}$ for all $x \in U$, $S({\triangledown}f) =  C({\triangledown}f)$. \qed
\end{proof}

Suppose $L$ is a lattice and $a,b \in L$. If there is a greatest element $z\in L$ such that $a \wedge z\leq b$, then this element $z$ is called 
the \emph{relative pseudocomplement of $a$ with respect to $b$} and is denoted by $a \Rightarrow b$. If $a \Rightarrow b$ exists, then it is unique. 
A \emph{Heyting algebra} $L$ is a lattice with $0$ in which $a \Rightarrow b$ exists for each $a,b \in L$. Heyting algebras are distributive lattices and 
any completely distributive lattice $L$ is a Heyting algebra in which 
\[ a \Rightarrow b = \bigvee \{ z \mid a \wedge z  \leq b\} .\]
Equationally Heyting algebras can be defined as lattices with $0$ and $\Rightarrow$ satisfying the identities \cite{BaDw74}:
\begin{enumerate}[({H}1)]
\item $x \wedge (x \Rightarrow y) = x \wedge y$,
\item $x \wedge (x \Rightarrow y) = x \wedge (x \wedge y \Rightarrow x \wedge z)$,
\item $z \wedge (x \wedge y \Rightarrow x) = z$.
\end{enumerate}

It is known \cite{Moisil63, Monteiro1980} that every three-valued {\L}ukasiewicz algebra forms a Heyting algebra where
\begin{equation} \label{eq:heyting_luka}
x \Rightarrow y = {\vartriangle}{\sim} x \vee y \vee ({\triangledown}{\sim} x \wedge {\triangledown} y).
\end{equation}

\begin{example} \label{exa:heyting}
The chain $\mathbf{3}$ is a Heyting algebra in which 
\[ a \Rightarrow b = \left \{
\begin{array}{ll}
1 & \mbox{if $a \leq b$},\\
b & \mbox{if $a > b$}. 
\end{array}
\right .
\]
Also ${\bf 3}^U$ is a Heyting algebra in which $\Rightarrow$ is defined pointwise:
\[ (f \Rightarrow g)(x) = f(x) \Rightarrow g(x) .\]
Since $\mathbf{3}^U$ and $\mathcal{A}(U)$ are isomorphic completely distributive lattices, 
$\mathcal{A}(U)$ is a Heyting algebra isomorphic to $\mathbf{3}^{U}$.

Let $x = (A,B)$ and $y = (C,D)$ be elements of $\mathcal{A}(U)$. We may use (\ref{eq:heyting_luka}) to infer $x \Rightarrow y$. Now
\begin{eqnarray*} 
 {\vartriangle}{\sim} x & = &  {\vartriangle}(B^c, A^c) = (B^c, B^c), \\
 {\triangledown}{\sim} x & = & {\triangledown}(B^c, A^c) = (A^c, A^c), \\
 {\triangledown}y & = & (D, D), \\
 {\triangledown}{\sim} x \wedge {\triangledown}y &=& (A^c \cap D, A^c \cap D),\\
 y \vee  ({\triangledown}{\sim} x \wedge {\triangledown}y) &=& (C \cup (A^c \cap D), D \cup (A^c \cap D)) = (C \cup (A^c \cap D), D),\\
 x \Rightarrow y &=& (B^c \cup C \cup (A^c \cap D), B^c \cup D).
\end{eqnarray*}
\end{example}

A De Morgan algebra $(L,\vee,\wedge,\sim,0,1)$ is a \emph{Kleene algebra} if for all $x,y\in L$,
\[
x \wedge {\sim} x \leq  y \vee {\sim} y.
\]
It is proved by Monteiro in \cite{Monteiro63} that every three-valued {\L}ukasiewicz algebra forms a Kleene algebra.
Note that $x \wedge {\sim} x  \leq u \leq y \vee {\sim} y$ for $x,y \in \mathbf{3}$. Obviously,  $\mathbf{3}^U$ and $\mathcal{A}(U)$ are isomorphic Kleene algebras
via $\varphi$. Note that already in \cite{KumBan2017} it is proved that $\mathbf{3}^U$ and $\wp(U)^{[2]} = \mathcal{A}(U)$ are isomorphic Kleene algebras. 

According to R.~Cignoli \cite{Cign86} a \emph{quasi-Nelson algebra} is defined as Kleene algebra $(A,\vee,\wedge,{\sim},0,1)$ where for each pair $a,b\in A$ the relative
pseudocomplement 
\begin{equation} \label{eq:nelson}
a \Rightarrow ({\sim} a \vee b)
\end{equation} 
exists. This means that every Kleene algebra whose underlying lattice is a Heyting algebra forms a quasi-Nelson algebra. In a quasi-Nelson algebra, the element 
(\ref{eq:nelson}) is denoted simply by $a\rightarrow b$. 

As shown by D.~Brignole and A.~Monteiro \cite{BrigMont1967}, the operation $\rightarrow$ satisfies the identities:
\begin{enumerate}[({N}1)]
\item $a \to a =1$, 
\item $({\sim} a \vee b) \wedge (  a \to b)  = {\sim} a \vee b$,
\item $a \wedge (a \to b)  = {\sim} a\vee b$,
\item $a \to (b\wedge c) = (a \to b)  \wedge (a \to c)$.
\end{enumerate}
A \emph{Nelson algebra} is a quasi-Nelson algebra satisfying the identity%

\begin{enumerate}[({N}5)]
\item $( a \wedge b)  \to c = a \to (b \to c)$.
\end{enumerate}
It is shown in \cite{BrigMont1967} that a Nelson algebra can be defined also as an algebra $(L,\vee,\wedge,\to,\sim,0,1)$, where
$(L,\vee,\wedge,\sim,0,1)$ is a Kleene algebra, and the binary operation $\to$ satisfies (N1)--(N5).  A Nelson algebra  is \emph{semisimple} if 
\begin{enumerate}[({N}6)]
\item $a \vee  (a \to 0) = 1$.
\end{enumerate}
It is known that  every three-valued {\L}ukasiewicz algebra defines a semisimple Nelson algebra by
setting 
\[ a \to b = {\triangledown}{\sim}a \vee b .\]
Similarly, each semisimple Nelson algebra defines a  three-valued {\L}ukasiewicz algebra by setting
\[ {\triangledown} a = {\sim} a \to 0 .\] 
In fact, the notions of  three-valued {\L}ukasiewicz algebra and semisimple Nelson algebra coincide  \cite{Monteiro1980}.

\begin{example} \label{exa:nelson}
The Kleene algebra defined on $\mathbf{3}$ forms also a Nelson algebra in which the operation $\to$ is defined as in the following table \cite{Rasiowa74}:
\[
\begin{tabular}{c|ccc}
$\to$ & $0$ & $u$ & $1$\\ \hline
$0$   & $1$ & $1$ & $1$ \\
$u$   & $1$ & $1$ & $1$ \\
$1$   & $0$ & $u$ & $1$
\end{tabular}
\]
The operation $\to$ is defined in $\mathbf{3}^U$ pointwise by $(f \to g)(x) = f(x) \to g(x)$. Note also that we can write 
\begin{equation} \label{eq:nelson_implication}
 (f \to g)(x) = f(x) \Rightarrow ({\sim} f(x) \vee g(x)) 
\end{equation}
It can be seen in the above table that the Nelson algebra $\mathbf{3}$ is semisimple. Therefore, also $\mathbf{3}^U$ forms a semisimple Nelson algebra. Because
$\mathbf{3}^U$ and $\mathcal{A}(U)$ are isomorphic as Heyting algebra (recall that if the operation $\Rightarrow$ exists, it is unique) and as Kleene algebras,
by (\ref{eq:nelson_implication}) we have that they are isomorphic also as semisimple Nelson algebras.

There are a couple of possibilities how we can derive the outcome of the operation $(A,B) \to (C,D)$ in $\mathcal{A}(U)$. We can either use (\ref{eq:nelson}) or
$a \to b = {\triangledown}{\sim}a \vee b$. It appears that the latter is simpler to apply here. For $(A,B), (C,D) \in \mathcal{A}(U)$, we have that
\[ {\sim} (A,B) = (B^c, A^c), \quad \triangledown (A,B) = (B,B), \quad \triangledown {\sim} (A,B) = (A^c, A^c).\]
Therefore,
\[ (A,B) \to (C,D) = (A^c \cup C, A^c \cup D).\] 
\end{example}

An algebra $(L, \vee, \wedge, {^*}, 0)$ is a $p$-algebra if $(L, \vee, \wedge, 0)$ is a bounded lattice and $^*$
is a unary operation on $L$ such that $x \wedge z = 0$ iff $z \leq x^*$. The element $x^*$ is the \emph{pseudocomplement} of $x$. 
It is well known that $x \leq y$ implies $x^* \geq y^*$. We also have for $x,y \in L$, 
\[
x^* = x^{***}, \quad (x \vee y)^* = x^* \wedge y^*, \quad (x \wedge y)^{**} = x^{**} \wedge y^{**}.
\]
Equationally $p$-algebras can be defined as lattices with $0$ such that the following identities hold \cite{Blyth}:
\begin{enumerate}[({P}1)]
\item $x \wedge (x \wedge y)^* = x \wedge y^*$,
\item $x \wedge 0^* = x$,
\item $0^{**} = 0$.
\end{enumerate}
Note that (P2) means that $0^*$ is the greatest element and we may denote it by $1$. Therefore, it is possible to include $1$ also
to the signature of a $p$-algebra.

An algebra $(L, \vee, \wedge, {^*}, {^+}, 0, 1)$ is a double $p$-algebra if $(L, \vee, \wedge , ^*, 0)$ is
a $p$-algebra and $(L, \vee, \wedge, {^+}, 1)$ is a dual $p$-algebra  (i.e.\, $z \geq x^+$ iff $x \vee z = 1$ for all $x,y \in L$).
The element $x^+$ is the \emph{dual pseudocomplement} of $a$. If $x \leq y$, then $x^+ \geq y^+$.
In addition, 
\[x^+ = x^{+++},\quad
(x \wedge y)^+ = x^+ \vee y^+,\quad 
(x \vee y)^{++} = x^{++} \vee y^{++}. 
\]
Note that by definition $x \leq x^{**}$ and $x^{++} \leq x$. Therefore, in a double $p$-algebra $x^{++} \leq x^{**}$.

\begin{example}\label{exa:3pseudo}
On $\mathbf{3}$ the operations $^*$ and $^+$ are defined as in the following table:
\begin{center}
\begin{tabular}{c|cc}
$x$ & $x^*$ & $x^+$\\ \hline
$0$ & $1$ & $1$ \\
$u$ & $0$ & $1$ \\
$1$ & $0$ & $0$
\end{tabular}
\end{center}
For a map $f \in \mathbf{3}^{U}$ the functions $f^*$ and $f^+$ are defined pointwise, that is,
\[  (f^*)(x) =   f(x)^* \quad \mbox{and}  \quad (f^+)(x)  = f(x)^+ .\]
\end{example}

\begin{lemma} \label{lem:support_star} 
If $f  \in \mathbf{3}^U$, then
\[ C(f^*) = S(f^*) = S(f)^c \qquad \mbox{and} \qquad  C(f^+) = S(f^+) = C(f)^c \] 
\end{lemma}

\begin{proof}
Let $x \in U$. Then,
\[
 x \in C(f^*)  \iff (f^*)(x) = 1 \iff f(x)^* = 1 \iff f(x) = 0 \iff x \notin S(f) \iff x \in S(f)^c.
\]
Because $f^*(x) \in \{0,1\}$ for all $x \in U$, $S(f^*) = C(f^*)$. Similarly,
\[
x \in C(f^+) \iff (f^+)(x) = 1 \iff f(x)^+ = 1 \iff f(x) \leq u \iff x \notin C(f)  \iff x \in C(f)^c.
\]
Since $f^+(x) \in \{0,1\}$ for all $x \in U$, $S(f^+) = C(f^+)$. \qed
\end{proof}

A \emph{pseudocomplemented De Morgan algebra} is an algebra $(L,\vee,\wedge,\sim,^{\ast},0,1)$ such that $(L,\vee,\wedge,\sim,0,1)$ is a
De Morgan algebra and $(L,\vee,\wedge,^{\ast},0,1)$ is a $p$-algebra. Such an algebra always forms a double $p$-algebra, where the pseudocomplement
operations determine each other by
\begin{equation} \label{eq:dual_pseudo}
\sim x^{\ast}=(\sim x)^{+}\mbox{ and }\sim x^{+}=(\sim x)^{\ast}.
\end{equation}

We say that a double $p$-algebra is \emph{regular} if it satisfies the condition
\begin{equation} \label{eq:regular} 
x^* = y^* \mbox{ and } x^+ = y^+ \mbox{ imply } x=y. 
\end{equation}
T.~Katri\v{n}\'{a}k \cite{Kat73} has shown that any regular double pseudocomplemented lattice forms a Heyting algebra such that
\begin{equation} \label{Eq:Katrinak}
a \Rightarrow b = (a^* \vee b^{**}) \wedge [(a\vee a^*)^{+} \vee a^*\vee b \vee b^*]. 
\end{equation}

A $p$-algebra $(L,\vee,\wedge,^*,0,1)$ is a \emph{Stone algebra} if $L$ is distributive and for all $x \in L$,
\begin{equation} \label{Eq:Stone}
x^{\ast}\vee x^{\ast\ast}=1.
\end{equation}
A \emph{double Stone algebra} is a distributive double $p$-algebra $(L, \vee, \wedge, {^*}, {^+}, 0, 1)$ satisfying 
(\ref{Eq:Stone}) and 
\begin{equation} \label{Eq:DualStone}
x^+ \wedge x^{++} = 0.
\end{equation}

\begin{example}\label{exa:double_stone}
As a distributive double $p$-algebra, $\mathbf{3}$ forms a double Stone algebra, because $x^*$ or $x^{**}$ equals 1 for any $x \in \mathbf{3}$,
and $x^+$ or $x^{++}$ is 0. From the table of Example~\ref{exa:3pseudo} we can see that (\ref{eq:regular}) holds in $\mathbf{3}$, meaning that
$\mathbf{3}$ is a regular double Stone algebra. This also implies that $\mathbf{3}^U$ forms a regular double Stone algebra.  

Because $\mathcal{A}(U)$ is isomorphic to $\mathbf{3}$, also $\mathcal{A}(U)$ is a double Stone algebra in which
\[ (A,B)^* = (B^c,B^c) \qquad \mbox{and} \qquad (A,B)^+ = (A^c,A^c). \]
\end{example}

It is known that every regular double Stone algebra $(L,\vee,\wedge,{^*},{^+},0,1)$ defines a three-valued {\L}ukasiewicz algebra 
$(L,\vee,\wedge,{\sim}, {\triangledown},0,1)$ by setting
\[
  {\triangledown} a = a^{**} \qquad \mbox{and} \qquad {\sim} a = a^* \vee (a \wedge a^+) .
\]
Similarly, each  three-valued {\L}ukasiewicz algebra defines a double Stone algebra by 
\[
 a^* = {\sim}{\triangledown}a \qquad \mbox{and} \qquad a^+ = {\triangledown}{\sim}a.
\]
These pseudocomplement operations determine each other by (\ref{eq:dual_pseudo}).
The correspondence between  regular double Stone algebras and   three-valued {\L}ukasiewicz algebras is one-to-one; see \cite{Boicescu91} for details and further references. 
Note that this means that also regular double Stone algebras and semi-simple Nelson algebras coincide.

\begin{example}
On $\mathbf{3}^{U}$ the operations $^{\ast},^{+},\rightarrow$ and ${\triangledown}$ can be defined as follows
in terms of the core and support of the functions, cf. Lemmas \ref{lem:support_triangle} and \ref{lem:support_star}.
For $f,g\in\mathbf{3}^{U}$,
\begin{eqnarray*}
& f^*(x) =  \left\{
\begin{array}{ll}
1 & \mbox{if $x\notin S(f)$,}\\
0 & \mbox{otherwise;}
\end{array} \right .
\qquad\qquad\qquad
& f^{+}(x) =  \left\{
\begin{array}{ll}
1 & \mbox{if $x \notin C(f)$,}\\
0 & \mbox{otherwise;}
\end{array} 
\right.
\\[2mm]
& ({\triangledown} f)(x) = \left\{
\begin{array}{ll}
1 &\mbox{if $x\in S(f)$,}\\
0 &\mbox{otherwise;}
\end{array} \right.
\qquad\qquad\qquad
& (f \rightarrow g)(x)  =  \left\{
\begin{array}{ll}
1    & \mbox{if $x\notin C(f)$,}\\
g(x) & \mbox{otherwise.}
\end{array}
\right. 
\end{eqnarray*}
\end{example}

The following proposition shows how in the presence of $\sim$, all operations 
$^*$, $^+$, ${\triangledown}$, ${\vartriangle}$, $\to$, $\Rightarrow$ are defined in terms of \emph{one} of them.

\begin{proposition} \label{Prop:ImplyLuka}
Let $\mathcal{F}$ be a complete polarity sublattice of $\mathbf{3}^U$. If $\mathcal{F}$ is closed with respect to at least one of the operations
$^*$, $^+$, ${\triangledown}$, ${\vartriangle}$, $\to$, $\Rightarrow$ defined in  $\mathbf{3}^U$, then  $\mathcal{F}$
is closed with respect to all these operations.
\end{proposition}

\begin{proof}
Let us first note that since $\mathcal{F}$ is a sublattice of $\mathbf{3}^U$, it is distributive. In addition, the least element $\bot$ and the greatest element
$\top$ of $\mathbf{3}^U$ are in $\mathcal{F}$, because $\mathcal{F}$ is a complete sublattice of $\mathbf{3}^U$. 

We have noticed that $^*$ and $^+$ fully determine each other in the presence of $\sim$ and together they define $\Rightarrow$ by (\ref{Eq:Katrinak}).
When $\mathcal{F}$ is closed under $^*$ or $^+$ of $\mathbf{3}^U$, the regularity condition (\ref{eq:regular}) holds and Stone identities
(\ref{Eq:Stone}) and (\ref{Eq:DualStone}) are valid. Hence, $\mathcal{F}$ forms a regular double Stone algebra.
We have seen that each regular double Stone algebra defines a semisimple Nelson algebra and a three-valued {\L}ukasiewicz algebra.
Therefore, if $\mathcal{F}$ is closed with respect to $^*$ or $^+$, it is closed with respect to all of the mentioned operations.

Similarly, ${\triangledown}$ and ${\vartriangle}$ define each other in terms of $\sim$. As we know, three-valued {\L}ukasiewicz algebras uniquely determine
semisimple Nelson algebras and regular double Stone algebras. Therefore, if $\mathcal{F}$ is closed with respect to ${\triangledown}$ or ${\vartriangle}$ of
$\mathbf{3}^U$, it is closed with respect to $^*$, $^+$, $\to$. Additionally, $\sim$, ${\triangledown}$ and ${\vartriangle}$ determine $\Rightarrow$ by (\ref{eq:heyting_luka}).

If $\mathcal{F}$ is closed with respect to $\to$ of $\mathbf{3}^U$, then it forms a semisimple Nelson algebra, which in turn defines uniquely a regular
double Stone algebra and a three-valued {\L}ukasiewicz algebra. Thus, $\mathcal{F}$ is closed with respect to all of the mentioned operations.

Finally, let $\mathcal{F}$ be closed with respect to $\Rightarrow$. Because $\bot \in \mathcal{F}$, $f^*$ is defined by $f \Rightarrow \bot$ for each
$f \in \mathcal{F}$. From this we get that  $\mathcal{F}$ is closed with respect to $^*$, $^+$, ${\triangledown}$, ${\vartriangle}$, $\to$. \qed
\end{proof}

We end this section by noting that the map $\varphi$ defined in Proposition~\ref{Prop:Correspondence} preserves all operations considered in this section.
Indeed, let $f \in \mathbf{3}^U$. We have already noted that $\varphi({\sim}f) = {\sim}\varphi(f)$. Now 
\[ \varphi(f^*) = (C(f^*),S(f^*)) = (S(f)^c,S(f)^c) = (C(f),S(f))^* = \varphi(f)^* .\]
As we have seen, the operations $^+$, ${\triangledown}$, ${\vartriangle}$, $\to$, $\Rightarrow$ can be defined in terms of $\vee$, $\wedge$,
$\sim$, and $^*$, so they are preserved with respect to $\varphi$.

\section{Alexandrov topologies defined by complete sublattices of $\mathbf{3}^U$}\label{sec:Alex}

An \emph{Alexandrov topology} \cite{Alex37,Birk37} $\mathcal{T}$ on $U$ is a topology in which also intersections of  open sets are open, 
or equivalently, every point $x \in U$ has the \emph{least neighbourhood} $N(x) \in \mathcal{T}$.
For an Alexandrov topology $\mathcal{T}$, the least neighbourhood  of $x$ is $N(x)= \bigcap \{ B \in \mathcal{T} \mid x \in B \}$. 
Each Alexandrov topology $\mathcal{T}$ on $U$ defines a quasiorder $\leq_\mathcal{T}$ on $U$ by $x \, \leq_\mathcal{T} \, y$ if and only if
$y \in N(x)$  for all $x, y \in U$. On the other hand, for a quasiorder $\leq$ on $U$, the set of all $\leq$-closed subsets of $U$ forms an Alexandrov topology
$\mathcal{T}_\leq$, that is, $B \in \mathcal{T}_\leq$ if and only if $x \in B$ and $x \leq y$ imply $y \in B$. 
Let $[ x ) = \{ y \in X \mid x \leq y \}$. In $\mathcal{T}_\leq$, $N(x) = [x)$ for any $x \in U$.  
The correspondences $\mathcal{T} \mapsto {\leq_\mathcal{T}}$ and ${\leq} \mapsto \mathcal{T}_\leq$ are
mutually invertible bijections between the classes of all Alexandrov topologies and of all quasiorders on the set $U$.

Let $\leq$ be a quasiorder on $U$. We denote its inverse by $\geq$. Obviously, also $\geq$ is a quasiorder and we denote its Alexandrov topology by
$\mathcal{T}_\geq$. We say that topologies $\mathcal{T}_1$ and $\mathcal{T}_2$ are \emph{dual} if
\[ X \in \mathcal{T}_1 \iff X^c \in \mathcal{T}_2 .\]
The topologies $\mathcal{T}_\leq$ and $\mathcal{T}_\geq$ are mutually dual. The smallest neighbourhood of a point $x \in U$ 
in $\mathcal{T}_\geq$ is $(x] = \{ y \in X \mid x \geq y \}$.

For the sake of completeness, we prove the following claim.

\begin{lemma} \label{lem:inverse_relation}
Let $\mathcal{T}_1$ and $\mathcal{T}_2$ be dual topologies and let $\leq_1$ and $\leq_2$ be the corresponding quasiorders, respectively.
Then ${\leq}_1 = {\geq}_2$. 
\end{lemma}

\begin{proof} Suppose $x \leq_1 y$, that is, $x \in \bigcap\{ Y \in \mathcal{T}_1 \mid y \in Y\}$. If $x \ngeq_2 y$, that is, $y \nleq_2 x$, then
$y \notin  \bigcap\{ X \in \mathcal{T}_2 \mid x \in X\}$. This means that there is $X \in \mathcal{T}_2$ such that $x \in X$ and $y \notin X$.
Because $\mathcal{T}_2$ is the dual topology of $\mathcal{T}_1$, then there is $Y \in \mathcal{T}_1$ such that $X = Y^c$. This means
that $y \in Y$ and $x \notin Y$. Therefore, $x \notin  \bigcap\{ Y \in \mathcal{T}_1 \mid y \in Y\}$, a contradiction. Thus,  $x \geq_2 y$ holds,
and  $x \leq_1 y$ implies  $x \geq_2 y$. Symmetrically we can show that $x \geq_2 y$ implies  $x \leq_1 y$, which completes the proof. \qed
\end{proof}

Let us now recall from \cite{JRV09} how Alexandrov topologies relate to rough set approximations. Let $\leq$ be a quasiorder on $U$. Then for any
$X \subseteq U$, 
\[ X^\UP = \{ x \in U \mid [x) \cap X \neq \emptyset\} \quad \mbox{and} \quad X^\DOWN = \{ x \in U \mid [x) \subseteq X \}.  \]
Let us denote $\wp(U)^\UP  = \{ X^\UP \mid X \subseteq U\}$ and $\wp(U)^\DOWN  = \{ X^\DOWN \mid X \subseteq U\}$.
Then,
\begin{equation} \label{eq:rough_from_alex}
\mathcal{T}_\leq = \wp(U)^\DOWN  \qquad \mbox{and}  \qquad \mathcal{T}_\geq = \wp(U)^\UP .
\end{equation}
In particular, $(x] = \{x\}^\UP$ for all $x \in U$.

\begin{remark} \label{rem:quasi_other_definitions}
There are also other choices for lower-upper approximation pairs defined in terms of a quasiorder $\leq$ on $U$. In \cite{KumBan2012,KumBan2015}
Kumar and Banerjee define the operators $\mathsf{L}$ and $\mathsf{U}$ by 
\[ 
\mathsf{L}(X) = \bigcup \{ D \in \mathcal{T}_\leq \mid D \subseteq X \} \quad \mbox{and} \quad
\mathsf{U}(X) = \bigcap \{ D \in \mathcal{T}_\leq \mid X \subseteq D \} \]
for any $X \subseteq U$. The sets $\mathsf{L}(X)$ and $\mathsf{U}(X)$ belong to the same topology $\mathcal{T}_\leq$, whose
elements are called `definable'. These operators can be also be written in form
\[ \mathsf{L}(X) = \{ x \in U \mid [x) \subseteq X\} \quad \mbox{and} \quad \mathsf{U}(X) = \{ x \in U \mid (x] \cap X \neq \emptyset \}. \]

This approach differs significantly from ours, because now the rough set system
\[ \mathcal{RS}' = \{ ( \mathsf{L}(X), \mathsf{U}(X) ) \mid X \subseteq U \} \]
is not generally a lattice with respect to the coordinatewise order.
Indeed, let $U = \{a,b,c\}$ and let the quasiorder $\leq$ on $U$ be defined by $[a) = \{a\}$, $[b) = \{a,b\}$ and $[c) = U$. The $\leq$-closed subsets
form an Alexandrov topology
\[ \mathcal{T}_\leq = \{ \emptyset, \{a\}, \{a,b\}, U \},  \]
and
\[\mathcal{RS}' = \{ (\emptyset,\emptyset), (\{a\}, \{a\}), (\emptyset,\{a,b\}), (\emptyset,U), (\{a,b\},\{a,b\}), (\{a\},U), (U,U) \}. \]
Now $(\{a\}, \{a\})$ and $(\emptyset,\{a,b\})$ have minimal upper bounds $(\{a,b\},\{a,b\})$ and $(\{a\},U)$, but
not a least one. Thus, $\mathcal{RS}'$ is not a lattice.

For a quasiorder $\leq$ on $U$, a pair of rough approximations is defined in \cite[Theorem~4]{Ganter2008} for any $X \subseteq U$ by
\[ \underline{X} = \{ x \in U \mid (\forall y \leq x) \, y \in X\}
\quad \mbox{and} \quad
\overline{X} = \{ x \in U \mid (\exists y \geq x) \, y \in X \} . \]
This means that
\[ \underline{X} = \{ x \in U \mid (x] \subseteq X\} \quad \mbox{and} \quad \overline{X} = \{ x \in U \mid [x) \cap X \neq \emptyset \} \]
for all $X \subseteq U$. Obviously, $\underline{X}$ and $\overline{X}$ belong to the same Alexandrov topology $\mathcal{T}_\geq$. This also means that
$\{ (\underline{X},\overline{X} ) \mid X \subseteq U \}$ does not necessarily form a lattice.
\end{remark}  
  
\begin{lemma} \label{lem:inducesAlex}
Let $\mathcal{F}$ be a complete sublattice of $\mathbf{3}^U$.
\begin{enumerate}[\rm (a)]
\item $C(\mathcal{F}) := \{ C(f) \mid f  \in \mathcal{F} \}$ and $S(\mathcal{F}) := \{ S(f) \mid f  \in \mathcal{F} \}$ are Alexandrov topologies on $U$.
\item If $\mathcal{F}$ is closed with respect to $\sim$, then $C(\mathcal{F})$ and $S(\mathcal{F})$ are dual.
\item If $\mathcal{F}$ is a three-valued {\L}ukasiewicz subalgebra of $\mathbf{3}^U$, then $C(\mathcal{F}) = S(\mathcal{F})$ is a Boolean lattice.
\end{enumerate}
\end{lemma}

\begin{proof}
(a) By Proposition~\ref{Prop:Correspondence}, the map $\varphi \colon f \mapsto (C(f), S(f))$ is an order-isomorphism between the
complete lattices  $\mathbf{3}^U$ and $\mathcal{A}(U)$. Because $\mathcal{F}$ is a complete sublattice of $\mathbf{3}^U$ its
$\varphi$-image is a complete sublattice of $\mathcal{A}(U)$. This means that $C(\mathcal{F})$ and $S(\mathcal{F})$ are closed
with respect to arbitrary unions and intersections. Thus, they are Alexandrov topologies.

(b) Suppose $\mathcal{F}$ is closed with respect to $\sim$. Then, by Lemma~\ref{lem:support_demorgan}, for $f \in \mathcal{F}$,
\[ C(f)^c = S({\sim}f) \in S(\mathcal{F}) \quad \mbox{and} \quad S(f)^c = C({\sim}f) \in C(\mathcal{F}). \]
Hence,  $C(\mathcal{F})$ and $S(\mathcal{F})$ are dual topologies.

(c) Since $\mathcal{F}$ is a three-valued {\L}ukasiewicz subalgebra of $\mathbf{3}^U$, it is also closed with respect to $^*$ and $^+$. 
By  Lemma~\ref{lem:support_star},
\[   S(f)^c = S(f^*) \quad \mbox{and} \quad C(f)^c =  C(f^+)\]
for any $f \in \mathcal{F}$. This implies that $S(\mathcal{F})$ and $C(\mathcal{F})$ are closed with respect to set-theoretical complement.
Because they are Alexandrov topologies, they form Boolean lattices. In addition,
\[ S(f) = S(f^*)^c = C({\sim}f^*) \quad \mbox{and}  \quad   C(f) = C(f^+)^c = S({\sim}f^+),\]
which implies that  $C(\mathcal{F}) = S(\mathcal{F})$. \qed
\end{proof}

Let $\mathcal{F}$ be a complete polarity sublattice of $\mathbf{3}^U$. Then, by Lemma~\ref{lem:inducesAlex}, $C(\mathcal{F})$ and
$S(\mathcal{F})$ are dual Alexandrov topologies. Let us define a binary relation $\leq_\mathcal{F}$ on $U$ by
\[
  x \leq_\mathcal{F} y \iff \mbox{$f(x) = 1$ implies $f(y) =1$ for all $f \in \mathcal{F}$}.
\]
Let us also introduce the following notation
\[ [ x ) _\mathcal{F} = \{y \in U \mid x \leq_\mathcal{F} y \}
\qquad \mbox{ and } \qquad
   (x]_\mathcal{F} = \{y \in U \mid x \geq_\mathcal{F} y \},
\]
where $\geq_\mathcal{F}$ is the inverse relation of $\leq_\mathcal{F}$.

\begin{lemma} \label{lem:properties_quasiorder}
Let $\mathcal{F}$ be a complete polarity sublattice of  $\mathbf{3}^U$.
\begin{enumerate}[\rm (a)]
\item The relation $\leq_\mathcal{F}$ is the quasiorder corresponding to the Alexandrov topology $C(\mathcal{F})$ and $[x) _\mathcal{F}$ 
is the smallest neighbourhood of the point $x$ in $C(\mathcal{F})$.
\item The relation $\geq_\mathcal{F}$ is the quasiorder corresponding to the Alexandrov topology $S(\mathcal{F})$ and $(x] _\mathcal{F}$
is the smallest neighbourhood of the point $x$ in $S(\mathcal{F})$.
\item $x \leq_\mathcal{F} y$ if and only if $f(x) = 0$ imply $f(y) = 0$ for all $f \in \mathcal{F}$.  
\end{enumerate}
\end{lemma}

\begin{proof} (a) Suppose that $x \leq_\mathcal{F} y$. By definition this is equivalent to that $x \in C(f)$ implies $y \in C(f)$ for all $f \in \mathcal{F}$.
From this we obtain $y \in \bigcap \{ C(f) \mid f \in \mathcal{F} \ \mbox{and} \ x \in C(f)\}$. 
This means that $y$ belongs to the smallest neighbourhood of $x$ in the Alexandrov topology $C(\mathcal{F})$.
On the other hand, if  $x \nleq_\mathcal{F} y$, then there exists $g \in \mathcal{F}$ such that $g(x) = 1$, but $g(y) \neq 1$.
This then means that $x \in C(g)$ and $y \notin C(g)$. From this we obtain
$y \notin \bigcap \{ C(f) \mid f \in \mathcal{F} \ \mbox{and} \ x \in C(f)\}$. We deduce that
$\leq_\mathcal{F}$ is the quasiorder corresponding to the Alexandrov topology $C(\mathcal{F})$. Obviously, $[ x ) _\mathcal{F}$
is the smallest neighbourhood of the point $x$ in $C(\mathcal{F})$. Claim (b) can be proved similarly.

(c) Assume  $x \leq_\mathcal{F} y$. Since  $\geq_\mathcal{F}$ is the quasiorder of the Alexandrov topology $S(\mathcal{F})$,  $y \geq_\mathcal{F} x$
means that  $x \in \bigcap \{ S(f) \mid f \in \mathcal{F} \ \mbox{and} \ f(y) \geq u \}$. Suppose that $f(x) = 0$. We must have
$f(y) \ngeq u$ which is equivalent $f(y) = 0$. On the other hand, if $x \nleq_\mathcal{F} y$, that is, $y \ngeq_\mathcal{F} x$, then there is
$g \in \mathcal{F}$ such that $g(y) \geq u$ and $g(x) \ngeq u$. The latter means $g(x) = 0$. Therefore, $g(x) = 0$ does not imply $g(y) = 0$. \qed
\end{proof}

\begin{remark}
Note that if $\mathcal{F}$ is a complete sublattice and a three-valued {\L}ukasiewicz subalgebra of $\mathbf{3}^U$, then the relation $\leq_\mathcal{F}$
is an equivalence. Indeed, suppose that $x \leq_\mathcal{F} y$. Then $y$ belongs to the smallest neighbourhood of $x$ in the Alexandrov topology
$C(\mathcal{F})$. Now $C(\mathcal{F}) = S(\mathcal{F})$ by Lemma~\ref{lem:inducesAlex}(c). 
This means that  $y$ belongs to the smallest neighbourhood of $x$ in the Alexandrov topology
$S(\mathcal{F})$, and therefore $x \geq_\mathcal{F} y$. Thus, $\leq_\mathcal{F}$ is symmetric.

It is also easy to see that if $\leq_\mathcal{F}$ is an equivalence and $x \leq_\mathcal{F} y$, then
$f(x) = f(y)$ for all $f \in \mathcal{F}$. Indeed, if $f(x) = 0$, then $f(y) = 0$ by Lemma~\ref{lem:properties_quasiorder}. Similarly,
$f(x) = 1$ implies $f(y) = 1$.  If $f(x) = u$, then $f(y)$ must be $u$, because $f(y) = 0$ or $f(y) = 1$ and $y \leq_\mathcal{F} x$ would
imply $f(x) = 0$ or $f(x) = 1$, a contradiction.
\end{remark}

The following lemma is now clear by (\ref{eq:rough_from_alex}).

\begin{lemma}
Let $\mathcal{F}$ be a complete polarity sublattice of  $\mathbf{3}^U$. If we define the operators $^\DOWN$ and $^\UP$ in terms of  $\leq_\mathcal{F}$,
then $C(\mathcal{F}) = \wp(U)^\DOWN$ \ and  \ $S(\mathcal{F}) = \wp(U)^\UP$.
\end{lemma}

\begin{example}
We have already noted that $\mathcal{A}(U)$ is isomorphic to 
$\mathbf{3}^U$ as a three-valued {\L}ukasiewicz algebra, as a  regular double Stone algebra and as a semi-simple Nelson algebra,
because $\varphi$ preserves all these operations.

Let us consider the three-element set $U = \{a,b,c\}$.
The set  $\mathbf{3}^{\{a,b,c\}}$ can be viewed as a set of 3-valued characteristic vectors of length 3, that is,
\[ \{ (x,y,z) \mid x,y,z \in \{0,u,1\} \} . \]
Obviously, there are 27 such vectors.  Let us agree that the first position corresponds to $a$, the second 
corresponds to $b$, and the third corresponds to $c$.

The operations in $\mathbf{3}^{\{a,b,c\}}$ are unique and are `lifted' pointwise from $\mathbf{3}$. This means that if $(x,y,z) \in \mathbf{3}^{\{a,b,c\}}$,
then 
\[ {\sim} (x,y,z) =  ( {\sim}x, {\sim}y, {\sim}z) \qquad \mbox{and} \qquad   (x,y,z)^* =  (x^*,y^*,z^*) , \]
for instance. 

\bigskip%

Let us consider a collection $\mathcal{RS} \subseteq \wp(U) \times \wp(U)$ such that 
\[ \mathcal{RS} = \{ (\emptyset,\emptyset), (\{a\}, \{a\}), (\emptyset, \{b,c\}), (\{a\}, U), (\{b,c\}, \{b,c\}), (U,U)\} , \]
which is the rough set system of the equivalence $E$ on $U$ having the equivalence classes $\{a\}$ and $\{b,c\}$.

The corresponding 3-valued functions are
\[
\begin{array}{lll}
  f_{(\emptyset,\emptyset)} = (0,0,0), & f_{ (\{a\}, \{a\})} = (1,0,0), &  f_{ (\emptyset, \{b,c\})} = (0,u,u), \\[2mm]
  f_{ (\{a\}, U)} = (1,u,u), & f_{ (\{b,c\}, \{b,c\})} = (0,1,1),  &  f_{ (U,U)} = (1,1,1).
\end{array}
\]
Let us denote the set of these functions by $\mathcal{F}$. Next we construct the Alexandrov topologies 
$C(\mathcal{F})$ and $S(\mathcal{F})$, and the relation $\leq_\mathcal{F}$. We will show that $\mathcal{F}$ forms
a three-valued {\L}ukasiewicz subalgebra of $\mathbf{3}^U$, and therefore  $\leq_\mathcal{F}$ is an equivalence and
$C(\mathcal{F}) = S(\mathcal{F})$ is a Boolean algebra.

It is easy to see that $\mathcal{F}$ is closed with respect to $\sim$ of $\mathbf{3}^U$:
\[
\begin{array}{lll}
 {\sim}(0,0,0) = (1,1,1), & {\sim}(1,0,0) =  (0,1,1),  &  {\sim}(0,u,u) =  (1,u,u),\\[2mm]
 {\sim}(1,u,u) = (0,u,u), & {\sim}(0,1,1) =  (1,0,0), &  {\sim}(1,1,1) =  (0,0,0).
\end{array}
\]
 Similarly,  $\mathcal{F}$ is closed with respect to $^*$:
\[
\begin{array}{lll} 
  (0,0,0)^* = (1,1,1), & (1,0,0)^* =  (0,1,1), & (0,u,u)^* = (1,0,0),\\[2mm]
  (1,u,u)^* = (0,0,0), & (0,1,1)^* =  (1,0,0), & (1,1,1)^* = (0,0,0).
\end{array}
\]
This means that  $\mathcal{F}$ forms a three-valued {\L}ukasiewicz subalgebra of $\mathbf{3}^U$. The possibility operation
$\triangledown$ is defined by
\[
\begin{array}{lll} 
  \triangledown (0,0,0) = (0,0,0), & \triangledown (1,0,0) =  (1,0,0), & \triangledown (0,u,u) = (0,1,1),\\[2mm]
  \triangledown (1,u,u) = (1,1,1), & \triangledown (0,1,1) =  (0,1,1), & \triangledown (1,1,1) = (1,1,1).
\end{array}
\]


Let us consider the set $S(\mathcal{F})$. Now
\[
\begin{array}{lll}
S(0,0,0) = \emptyset, &  S(1,0,0) =  \{a\},   &  S(0,u,u) =  \{b,c\}, \\[2mm]
S(1,u,u) = U,         &  S(0,1,1) =  \{b,c\}, & S(1,1,1) =  U.
\end{array}
\]

This means that 
\[ S(\mathcal{F}) = \{ \emptyset, \{a\}, \{b,c\}, U \} ,\]
and this topology also is equal to $C(\mathcal{F})$.
The topology  $S(\mathcal{F})$ induces an equivalence $\leq_\mathcal{F}$ whose equivalence classes are  $\{a\}$ and $\{b,c\}$.
Obviously, the rough set system defined by 
$\leq_\mathcal{F}$ coincides with $\mathcal{RS}$ above.
\end{example}

The above example shows how for each equivalence $E$ on $U$, we obtain a three-valued {\L}ukasiewicz subalgebra $\mathcal{F}$ of $\mathbf{3}^U$ such that in terms of 
$\mathcal{F}$ we can construct the same equivalence $E$ that we started with. On the other hand, we know from the literature \cite{Comer95} that
for each complete atomic regular double Stone algebra $\mathbb{A}$ there exists a set $U$ and an equivalence $E$ on $U$ such that the rough set system determined
by $E$ is isomorphic to $\mathbb{A}$. As we have noted, regular double Stone algebras correspond to three-valued {\L}ukasiewicz algebras. Note that an ordered set
with a least element 0 is \emph{atomic} if every nonzero element has an atom a below it. 

Let us consider the two-element set $U = \{a,b\}$. Because $\mathbf{3}^U$ is finite, it is atomic. If $\mathcal{F}$ is a three-valued {\L}ukasiewicz subalgebra of
$\mathbf{3}^U$, then $\mathcal{F}$ is isomorphic to the rough set algebra determined by an equivalence $E$ on \emph{some} set, not necessarily $U$.
For instance, we can see that $\mathbf{3}^{\{a,b\}}$ has 6 different  three-valued {\L}ukasiewicz subalgebras: $\mathbf{2}$, $\mathbf{3}$,
$\mathbf{2} \times \mathbf{2}$, $\mathbf{2} \times \mathbf{3}$, $\mathbf{3} \times \mathbf{2}$,
and $\mathbf{3} \times \mathbf{3}$, but in $U$ it is possible to define only 2 equivalences: the universal relation and the diagonal relation. 
Therefore, not all complete three-valued {\L}ukasiewicz subalgebras $\mathcal{F}$ of $\mathbf{3}^{\{a,b\}}$ are such that $\mathcal{A}(\mathcal{F})$
is equal to a rough set system defined by an equivalence on $U$.

We can ask the following question:

\begin{question} \label{q:equivalence}
Which three-valued {\L}ukasiewicz subalgebras $\mathcal{F}$ of $\mathbf{3}^U$ are such that there is an equivalence $E$ on $U$ whose rough set system is equal to
$\mathcal{A}(\mathcal{F})$?
\end{question}

In \cite{JarRad11} we proved that if $\mathbb{A}$ is a Nelson algebra defined on an algebraic lattice, then there exists a set $U$ and a quasiorder $\leq$ on $U$
such the rough set Nelson algebra defined by $\leq$ is isomorphic to $\mathbb{A}$. Recall that an \emph{algebraic lattice} is a complete lattice such that every
element is a join of compact elements. A similar question can be also addressed for Nelson algebras:

\begin{question} \label{q:quasiorder}
Which Nelson subalgebras $\mathcal{F}$ of $\mathbf{3}^U$ are such that there is a quasiorder $\leq$ on $U$ whose rough set system equals  
$\mathcal{A}(\mathcal{F})$?
\end{question}

\section{Rough sets defined in terms of three-valued functions} \label{sec:main}

Next our aim is to answer Questions \ref{q:equivalence} and \ref{q:quasiorder}. Let $\mathcal{F} \subseteq \mathbf{3}^U$ and $x \in U$.
We define two functions $U \to \mathbf{3}$ by
\[
f^x = \bigwedge \{ f \in \mathcal{F} \mid f(x) = 1 \} \quad \mbox{ and } \quad f_x = \bigwedge \{ f \in \mathcal{F} \mid f(x) \geq u\}.
\]
In addition, we define an equivalence $\Theta$ on $\mathcal{F}$ as the kernel of $C$, that is,
\[ f  \Theta  g \iff C(f) = C(g). \]

\begin{lemma} \label{lem:neighbourhoods}
Let $\mathcal{F}$ be a complete polarity sublattice of $\mathbf{3}^U$. For all $x,y \in U$,
\begin{enumerate}[\rm (a)]
\item $f_x \leq f^x$;
\item $[ x ) _\mathcal{F} = C(f^x)$ and $ (x] _\mathcal{F} = S(f_x)$;
\item $x \leq_\mathcal{F} y \iff f_x \leq f_y \iff f^x \geq f^y$;
\item $f^x = \bigwedge \{ h \in \mathcal{F} \mid h \Theta  f_x \}$.
\end{enumerate}
\end{lemma}

\begin{proof} (a) Since $\{f \in \mathcal{F} \mid f(x) = 1\} \subseteq \{f \in \mathcal{F} \mid f(x) \geq u\}$, we have
\[ f^x = \bigwedge \{f \in \mathcal{F} \mid f(x) = 1\} \geq \bigwedge \{f \in \mathcal{F} \mid f(x) \geq u\} = f_x. \]

(b) Using Lemmas \ref{lem:ContinuousMeetsJoins} and \ref{lem:properties_quasiorder},
\begin{eqnarray*}
[x)_\mathcal{F} &=& \bigcap \{ C(f) \mid f \in \mathcal{F} \ \mbox{ and } \ x \in C(f) \} = C \big ( \bigwedge \{ f \in \mathcal{F} \mid x \in C(f)\} \big) \\
& = & C \big ( \bigwedge \{ f \in \mathcal{F} \mid f(x) = 1\} \big ) = C(f^x) 
\end{eqnarray*}
and
\[
(x]_\mathcal{F} = \bigcap \{ S(f) \mid f \in \mathcal{F} \ \mbox{ and } \ x \in S(f) \} = S \big ( \bigwedge \{ f \in \mathcal{F} \mid x \in S(f)\} \big ) \\
= S \big ( \bigwedge \{ f \in \mathcal{F} \mid f(x) \geq u\} \big ) = S(f_x). 
\]

(c) If $x \leq_\mathcal{F} y$, then $x \in (y]_\mathcal{F} = S(f_y)$ and $y \in [x)_\mathcal{F} = C(f^x)$. Firstly, $x \in S(f_y)$ means that $f_y(x) \geq u$
and $f_y \in \{ f \in \mathcal{F} \mid f(x) \geq u\}$ gives $f_x = \bigwedge  \{ f \in \mathcal{F} \mid f(x) \geq u\} \leq f_y$. Secondly, by
$y \in C(f^x)$ we have $f^x(y) = 1$ and $f^x \in  \{ f \in \mathcal{F} \mid f(y) = 1 \}$. From this we obtain
$f^x \geq \bigwedge \{ f \in \mathcal{F} \mid f(y) = 1 \} = f^y$.

On the other hand, by Proposition~\ref{Prop:Correspondence}, $f_x \leq f_y$ implies $x \in S(f_x) \subseteq S(f_y) = (y]_\mathcal{F}$ and hence
$x \leq_\mathcal{F} y$. Similarly,  $f^x \geq f^y$ implies $y \in C(f^y) \subseteq C(f^x) = [x)_\mathcal{F}$ and $x \leq_\mathcal{F} y$.

(d) Because $f^x \in \{ h \in \mathcal{F} \mid h \Theta f^x\}$, we have
\[ \bigwedge \{ h \in \mathcal{F} \mid h \Theta f^x\} \leq f^x. \]
Since $x \in [x)_\mathcal{F} = C(f^x)$, we have that $h \Theta f^x$ implies $x \in C(h)$, whence $h(x) = 1$. Therefore,
\[
\{ h \in \mathcal{F} \mid h \Theta f^x \} \subseteq \{ h \in \mathcal{F} \mid h(x) = 1\}.\] 
This yields
\[ f_x = \bigwedge \{ h \in \mathcal{F} \mid h(x) = 1\} \leq \bigwedge \{ h \in \mathcal{F} \mid h \Theta f^x \}, \]
completing the proof. \qed
\end{proof}

The following lemma describes the rough approximations in terms of cores and supports of maps.

\begin{lemma} \label{lem:approximations_new}
Let $\mathcal{F}$ be a complete polarity sublattice of  $\mathbf{3}^U$. If we define the operators $^\DOWN$ and $^\UP$ in terms of  $\leq_\mathcal{F}$,
then for any $X \subseteq U$,
\begin{enumerate}[\rm (a)]
\item $X^\UP = S(\bigwedge \{ f \in \mathcal{F} \mid X \subseteq S(f) \}$;
\item $X^\DOWN = C(\bigvee \{ f \in \mathcal{F} \mid C(f) \subseteq X \}$.
\end{enumerate}
\end{lemma}

\begin{proof} (a) The set $X^\UP \in S(\mathcal{F})$ is the smallest set in $S(\mathcal{F})$ containing $X$. On the other hand,
\[ \bigcap \{ S(f) \mid f \in \mathcal{F} \ \mbox{and} \ X \subseteq S(f) \} \]
is the smallest set in $S(\mathcal{F})$ containing $X$. We have that
\[ X^\UP = \bigcap \{ S(f) \mid f \in \mathcal{F} \ \mbox{and} \ X \subseteq S(f) \}
= S \big (\bigwedge \{ f \in \mathcal{F} \mid X \subseteq S(f) \} \big ). \]

(b) Similarly, $X^\DOWN \in C(\mathcal{F})$ is the greatest element of $C(\mathcal{F})$ included in $X$. Hence,
\[ X^\DOWN = \bigcup \{ C(f) \mid f \in \mathcal{F} \ \mbox{and} \ C(f) \subseteq X  \} =  C(\bigvee \{ f \in \mathcal{F} \mid C(f) \subseteq X \}.\] \qed
\end{proof}

Let $\mathcal{F}$ be a complete polarity sublattice of $\mathbf{3}^U$ and $x \in U$. We say that an element $x \in U$
is an \emph{$\mathcal{F}$-singleton} if $[x)_\mathcal{F} = \{x\}$. The following lemma gives a characterisation of
$\mathcal{F}$-singletons.

\begin{lemma} \label{lem:singletons}
Let $\mathcal{F}$ be a complete polarity sublattice of $\mathbf{3}^U$. An element $x \in U$ is
an $\mathcal{F}$-singleton if and only if there is a map $f \in \mathcal{F}$ with $C(f) = \{x\}$.
\end{lemma}

\begin{proof} By Lemma~\ref{lem:neighbourhoods}, $C(f^x) = [x)_\mathcal{F}$. If $x$ is a $\mathcal{F}$-singleton, then $C(f^x) = \{x\}$.
Conversely, suppose that there is a map $f \in \mathcal{F}$ such that $C(f) = \{x\}$. Because $f(x) = 1$, we have $f^x \leq f$. This implies
\[ \{x\} \subseteq C(f^x) \subseteq C(f) = \{x\}.\]
Thus, $C(f^x) = \{x\}$ and hence $x$ is an $\mathcal{F}$-singleton. \qed
\end{proof}

An element $x$ of a complete lattice $L$ is \emph{completely join-irreducible} if $x = \bigvee S$ implies $x \in S$.
Let us denote by $\mathcal{J}(L)$ the set of completely join-irreducible elements of $L$. A lattice $L$ is \emph{spatial}
if each of its elements is a join of completely join-irreducible elements. 

\begin{proposition} \label{prop:spatial}
Let $\mathcal{F}$ be a complete polarity sublattice of $\mathbf{3}^U$ and $x \in U$. Then $\mathcal{F}$ is spatial
and
\[ \mathcal{J(F)} = \{ f_x \mid x \in U\} \cup \{ f^x \mid x \in U\} .\]
\end{proposition}

\begin{proof}
The powerset $\wp(U)$ forms an algebraic lattice in which finite subsets of $U$ are the compact elements.
A product of algebraic lattices is algebraic (see \cite[Proposition~\mbox{I-4.12}]{gierz2003}),
which implies that $\wp(U) \times \wp(U)$ is algebraic. A complete sublattice of an algebraic lattice is algebraic \cite[Exercise~{7.7}]{DaPr02}.
Because $\mathcal{A(F)}$ is a complete sublattice of $\wp(U) \times \wp(U)$, $\mathcal{A(F)}$ is algebraic. We have already noted in Section~\ref{Sec:Three-valued}
that $\mathcal{A}(U)$ is completely distributive. It is known (see e.g.\@ \cite{JarRad11} and the references therein) that every algebraic and completely distributive
lattice is spatial. Thus, $\mathcal{A(F)}$ is spatial and because $\mathcal{F}$ is isomorphic to $\mathcal{A(F)}$, also $\mathcal{F}$ is spatial.

Next we need to find the set of completely join-irreducible elements of $\mathcal{F}$. First we show that each $f^x$
is join-irreducible. Suppose that $f^x = \bigvee \mathcal{G}$ for some $\mathcal{G} \subseteq \mathcal{F}$. Because
$f^x = \bigwedge \{ f \in \mathcal{F} \mid f(x) = 1\}$, $f^x(x) = \bigwedge \{ f(x) \in \mathcal{F} \mid f(x) = 1\} = 1$.
Since $(\bigvee \mathcal{G})(x) = 1$ and $\mathbf{3}$ is a chain, we have that $g(x) = 1$ for some $g \in \mathcal{G}$.
We obtain $g \in \{ f \in \mathcal{F} \mid f(x) = 1\}$ and $f^x = \bigwedge \{ f \mid \mathcal{F} \mid f(x) = 1\} \leq g$.
On the other hand $f^x = \bigvee \mathcal{G}$ gives that $f^x \geq g$. Hence, $f^x = g \in \mathcal{G}$ and $f^x$ is
completely join-irreducible. In an analogous way, we may show that also $f_x$ is completely irreducible.

It is clear that any $f \in \mathcal{F}$ is an upper bound of
\[ \mathcal{H} = \{ f_x \mid f_x \leq f\} \cup \{f^x \mid f^x \leq f\} .\]
Let $g$ be an upper bound of $\mathcal{H}$. We prove that $f \leq g$. For this, we assume that $f \nleq g$. This means
that there is an element $a \in U$ such that $f(a) \nleq g(a)$. Because $\mathbf{3}$ is a chain, we have that
$f(a) > g(a)$. We have now three possibilities.

(i) If $f(a) = 1$ and $g(a) = u$, then $f^a \leq f$, but $f^a(a) = 1$ and $g(a) = u$. Then $g$ is not an upper bound of $\mathcal{H}$,
a contradiction. Case (ii), when $f(a) = 1$ and $g(a) = 0$ is similar.

(iii) If $f(a) = u$ and $g(a) = 0$, then $f_a \leq f$, $f_a(a) \geq u$, and $g(a) = 0$. Thus, $g$ is not an upper bound
of $\mathcal{H}$, a contradiction. Since each case (i)--(iii) leads to a contradiction, we have that $f \leq g$ and $f$ is
the least upper bound of $\mathcal{H}$. We have that
\[ f = \bigvee \{ f_x \mid f_x \leq f\} \vee \bigvee \{f^x \mid f^x \leq f\} .\]
From this it directly follows also that
\[ \mathcal{J(F)} = \{ f_x \mid f_x \leq f\} \cup \{f^x \mid f^x \leq f\} . \] \qed
\end{proof}

Let us now introduce the following three conditions for a complete polarity sublattice $\mathcal{F}$ of $\mathbf{3}^U$. 
\begin{enumerate}[({C}1)]
\item If $x$ is an $\mathcal{F}$-singleton, then $x \in S(f)$ implies $x \in C(f)$ for all $f \in \mathcal{F}$.
\item For any $x$, we have $C(f_x) \subseteq \{x\}$.
\item For any $f,g \in \mathcal{F}$, $C(f) \subseteq S(g)$ implies $S ( \bigwedge \{ h \in \mathcal{F} \mid h \Theta f \}) \subseteq S(g)$.
\end{enumerate}

Let $\mathcal{F}$ be a complete polarity sublattice $\mathbf{3}^U$. If $\mathcal{A(F)} = \mathcal{RS}$ for some quasiorder $\leq$ on $U$,
then for each rough set $(X^\DOWN,X^\UP)$, there is a map $f \in \mathcal{F}$ such that $(C(f),S(f))$ equals $(X^\DOWN,X^\UP)$. Moreover,
for any $f \in \mathcal{F}$, the pair $(C(f),S(f))$ is in $\mathcal{RS}$. We also have that the Alexandrov topologies coincide, meaning that
$C(\mathcal{F}) = \wp(U)^\DOWN$ and $S(\mathcal{F}) = \wp(U)^\UP$. Because there is a one-to-one correspondence between Alexandrov topologies,
and $\leq$ is the quasiorder corresponding to the Alexandrov topology $\wp(U)^\DOWN$ and $\leq_\mathcal{F}$ is the quasiorder of $C(\mathcal{F})$,
we have that $\leq$ and $\leq_\mathcal{F}$ are equal. This means that rough set pairs operations can be defined
in two ways: either in terms of the rough approximations defined by the quasiorder $\leq_\mathcal{F}$ or in terms of the approximation pairs of
the maps in $\mathcal{F}$.

\begin{proposition} \label{prop:necessary}
Let $\mathcal{F}$ be a complete polarity sublattice $\mathbf{3}^U$. If $\mathcal{A(F)} = \mathcal{RS}$ for some quasiorder $\leq$,
then (C1)--(C3) hold.
\end{proposition}

\begin{proof} (C1) Let $x$ be an $\mathcal{F}$-singleton and $f \in \mathcal{F}$. There is $X \subseteq U$ such that
$C(f) = X^\DOWN$ and $S(f) = X^\UP$. Because $[x) = \{x\}$, $x \in S(f) = X^\UP$ means that $\{x\} \cap X = \emptyset$ and $x \in X$.
We have $[x) = \{x\} \subseteq X$, that is, $x \in X^\DOWN = C(f)$.    

(C2) By Lemma~\ref{lem:neighbourhoods}(b), $S(f_x) = (x]_\mathcal{F} = (x] = \{x\}^\UP$. Suppose $(Z^\DOWN,Z^\UP) \in \mathcal{RS}$
is such that $Z^\UP = \{x\}^\UP$. There is a map $g \in \mathcal{F}$ such that $(C(g),S(g)) = (Z^\DOWN,Z^\UP)$. Since
$x \in \{x\}^\UP = Z^\UP = S(g)$, we get $g(x) \geq u$. Therefore,
\[ f_x = \bigwedge \{ f \in \mathcal{F} \mid f(x) \geq u \} \leq g.\]
By the isomorphism given in  Proposition~\ref{Prop:Correspondence},  $(C(f_x),S(f_x)) \leq (C(g),S(g))$.
This means that $(C(f_x),S(f_x))$ is the least rough set such that the second component is $\{x\}^\UP$. Because
$(\{x\}^\DOWN, \{x\}^\UP)$ is such a rough set too, we have $C(f_x) \subseteq \{x\}^\DOWN \subseteq \{x\}$.

(C3) Assume $C(f) \subseteq S(g)$ for some $f,g \in \mathcal{F}$. We have that there are subsets $X,Y \subseteq U$ such that
$C(f) = X^\DOWN$ and $S(f) = Y^\UP$. Let us denote
  \[ f_\Theta = \bigwedge \{ h \in \mathcal{F} \mid h \Theta f \} .\]
Suppose that $(Z^\DOWN,Z^\UP) \in \mathcal{RS}$ is a rough set such that $Z^\DOWN = X^\DOWN$. Thus, there is $f' \in \mathcal{F}$ that satisfies
$(C(f'), S(f')) = (Z^\DOWN,Z^\UP)$. Because $f' \in \{ h \in \mathcal{F} \mid h \Theta f \}$, we have $f_\Theta \leq f'$.
By Proposition~\ref{Prop:Correspondence}, $(C(f_\Theta),S(f_\Theta)) \leq (C(f'),S(f'))$. Furthermore,
\[
  C(f_\Theta)  = C \big ( \bigwedge \{ h \in \mathcal{F} \mid h \Theta f \} \big ) 
   = \bigcap \{ C(h) \mid f \in \mathcal{F} \ \mbox{and} \ C(h) = C(f) \} = C(f).
\]
We have shown that $(C(f_\Theta),S(f_\Theta))$ is the smallest rough set such that its first component equals $X^\DOWN$.

Let $(A^\DOWN, A^\UP)$ be a rough set such that $A^\DOWN = X^\DOWN$. Then $X^\DOWN \subseteq A$ gives $X^{\DOWN \UP} \subseteq A^\UP$.
Note that $( (X^\DOWN)^\DOWN, (X^\DOWN)^\UP) ) = ( X^\DOWN, X^{\DOWN \UP} )$ is a  rough set. Therefore, $( X^\DOWN, X^{\DOWN \UP} )$ is the
smallest rough set such that its first component is $X^\DOWN$. Hence, $( X^\DOWN, X^{\DOWN \UP} ) = (C(f_\Theta), S(f_\Theta))$.
Since $X^\DOWN \subseteq Y^\UP$, we obtain
\[ S(f_\Theta) = X^{\DOWN \UP} \subseteq Y^{\UP \UP} = Y^\UP = S(g), \]
which completes the proof. \qed
\end{proof}

Let $\mathcal{F}$ be a complete polarity sublattice of $\mathbf{3}^U$. In the following theorem, we denote the rough approximations
defined by the quasiorder $\leq_\mathcal{F}$ by $X^\DOWN$ and $X^\UP$ for any $X \subseteq U$. Similarly, $\mathcal{RS}$ denotes the
corresponding rough set system.

\begin{theorem} \label{Thm:main}
Let $\mathcal{F}$ be a complete polarity sublattice of $\mathbf{3}^U$.
\begin{enumerate}[\rm (a)]
\item If (C1) holds, then $(C(f),S(f)) \in \mathcal{RS}$ for every $f \in \mathcal{F}$.
\item If (C2) and (C3) hold, then for any $(X^\DOWN,X^\UP) \in \mathcal{RS}$, there is $f \in \mathcal{F}$ such that
      $(X^\DOWN,X^\UP) = (C(f), S(f))$.
\end{enumerate}
\end{theorem}

\begin{proof} (a) Take $f \in \mathcal{F}$. We know that $C(f) \subseteq S(f)$, $C(f) \in \wp(U)^\DOWN$ and $S(f) \in \wp(U)^\UP$ by
Lemma~\ref{lem:properties_quasiorder}.
Let $x$ be an $\mathcal{F}$-singleton. By (C1), $x \in C(f) \cup S(f)^c$. We proved in \cite[Prop.~4.2]{JaPaRa13} that for a quasiorder $\leq$,
a pair $(A,B)$ is a rough set if and only if $A \in \wp(U)^\DOWN$, $B \in \wp(U)^\UP$, $A \subseteq B$ and $x \in A \cup B^c$ for all $x \in U$ such
that $[x) = \{x\}$. The claim follows directly from this.

(b) In \cite[Thm.~5.2]{JRV09}, we proved that for a quasiorder $\leq$ on $U$, 
\[ \mathcal{J(RS)} = \{ (\emptyset,\{x\}^\UP) \mid \, | [x) | \geq 2\} \cup \{ ([x),[x)^\UP) \mid x \in U\} \]
is the set of completely join-irreducible elements and each element of $\mathcal{RS}$ is a join of some (or none) elements of $\mathcal{J(RS)}$.

Let $x \in U$ be such that $|[x)| \geq 2$. Condition (C2) yields $C(f_x) \subseteq \{x\}$. Now $C(f_x) = \{x\}$ is not possible,
because Lemma~\ref{lem:singletons} would imply that $x$ is an $\mathcal{F}$-singleton, contradicting $|[x)| \geq 2$. We have $C(f_x) = \emptyset$ and
we have earlier noted that $S(f_x) = (x] = \{x\}^\UP$. Thus, $(\emptyset,\{x\}^\UP) = (C(f_x), S(f_x))$.    

\smallskip%

Let us next consider a rough set of the form $([x),[x)^\UP)$, where $x \in U$. 
Because $[x) = C(f^x) \subseteq S(f^x)$, $f^x$ is an element of
$\{ f \in \mathcal{F} \mid [x) \subseteq S(f) \}$. We obtain
\[ \bigwedge \{ f \in \mathcal{F} \mid [x) \subseteq S(f) \} \leq f^x \]
and further
\begin{equation} \label{eq:first_side}  
S \big ( \bigwedge \{ f \in \mathcal{F} \mid [x) \subseteq S(f) \} \big ) \leq S(f^x) . 
\end{equation}
Using Lemma~\ref{lem:approximations_new}, we obtain
\[ C(f^x) = [x) \subseteq  [x)^\UP = S \big ( \bigwedge \{ f \in \mathcal{F} \mid [x) \subseteq S(f) \} \big )  .\]      
By (C3), 
\[ S \big ( \bigwedge \{ h \in \mathcal{F} \mid h \Theta f^x \} \big ) \subseteq S \big ( \bigwedge \{ f \in \mathcal{F} \mid [x) \subseteq S(f) \} \big )  .
\]
We have proved in Lemma~\ref{lem:neighbourhoods} that $f^x = \bigwedge \{ h \in \mathcal{F} \mid h \Theta f^x\}$. This gives
$S(f^x) = S(\bigwedge \{ h \in \mathcal{F} \mid h \Theta f^x\})$ and we have
\begin{equation} \label{eq:other_side}
S(f^x) \subseteq  S \big ( \bigwedge \{ f \in \mathcal{F} \mid [x) \subseteq S(f) \} \big ).
\end{equation}
Combining (\ref{eq:first_side}) and (\ref{eq:other_side}), we have $[x)^\UP = S ( \bigwedge \{ f \in \mathcal{F} \mid [x) \subseteq S(f) \} = S(f^x)$.
Since $[x) = C(f^x)$,  $([x),[x)^\UP) = (C(f^x), S(f^x))$.

\smallskip%

Let $(X^\DOWN,X^\UP) \in \mathcal{RS}$. As we already noted, each element of $\mathcal{RS}$ is a join of elements of $\mathcal{J(RS)}$, that is,
\[ (X^\DOWN,X^\UP) = \bigvee_{i \in I} (J_i^\DOWN, J_i ^\UP) \]
for some $\{ (J_i^\DOWN, J_i ^\UP) \mid i \in I\} \subseteq \mathcal{J(RS)}$. By the above, every $(J_i^\DOWN, J_i ^\UP)$ is of the form
$(C(\varphi_i), S(\varphi_i))$, where each $\varphi_i$ belongs to $\mathcal{F}$. We have
\[
(X^\DOWN,X^\UP) = \bigvee_{i \in I} (J_i^\DOWN, J_i ^\UP) = \Big (\bigcup_{i \in I} J_i^\DOWN, \bigcup_{i \in I} J_i ^\UP \Big ) =
\Big (\bigcup_{i \in I} C(\varphi_i), \bigcup_{i \in I} S(\varphi \Big ) \\
= \Big ( C \big (\bigvee_{i \in I} \varphi_i \big), S \big ( \bigvee_{i \in I} \varphi_i \big ) \Big ), 
\]
completing the proof. \qed
\end{proof}

We can now write the following theorem answering to Question~\ref{q:quasiorder}.

\begin{theorem} \label{thm:quasiorder}
If $\mathcal{F} \subseteq \mathbf{3}^U$, then $\mathcal{A(F)} = \mathcal{RS}$ for some quasiorder on $U$ if and
only if $\mathcal{F}$ is a complete polarity sublattice of $\mathbf{3}^U$ satisfying (C1)--(C3).
\end{theorem}

\begin{proof}
Suppose that $\mathcal{RS} = \mathcal{A(F)}$. Then $\mathcal{A(F)} = \{ (C(f), S(f)) \mid f \in \mathcal{F}\}$ is a complete polarity
sublattice of $\wp(U) \times \wp(U)$. Let $\mathcal{G} \subseteq \mathcal{F}$. By Lemma~\ref{lem:ContinuousMeetsJoins},
\[ C \big ( \bigvee \mathcal{G} \big ) = \bigcup_{f \in \mathcal{F}} C(f) \quad \mbox{and} \quad
S \big (\bigvee \mathcal{G} \big ) = \bigcup_{f \in \mathcal{F}} S(f) .\]
Thus, 
\[ \mathcal{A}\big ( \bigvee \mathcal{G} \big ) = \big ( \bigcup_{f \in \mathcal{F}} C(f), \bigcup_{f \in \mathcal{F}} S(f) \big ) \in \mathcal{A(F)} .\] 
Using the inverse $\varphi^{-1}$ of the isomorphism $\varphi$ of Proposition~\ref{Prop:Correspondence}, we have
\[ \bigvee \mathcal{G} = \varphi^{-1}\big ( \mathcal{A}\big ( \bigvee \mathcal{G} \big ) \big ) \in \mathcal{F}.\]
Similarly, we can show that $\bigwedge \mathcal{G}$ belongs to $\mathcal{F}$. For $f \in \mathcal{F}$,
\[ \mathcal{A}({\sim}f) = ( C({\sim} f), S({\sim} f)) = (S(f)^c,C(f)^c) = {\sim} ( C(f), S(f)) \in \mathcal{A(F)}.\]
We have that ${\sim}f = \varphi^{-1}(\mathcal{A}({\sim}f))$ belongs to $\mathcal{F}$. Thus, $\mathcal{F}$ is a complete polarity sublattice
of  $\mathbf{3}^U$. By Proposition~\ref{prop:necessary}, $\mathcal{F}$ satisfies (C1)--(C3).

On the other hand, if $\mathcal{F}$ is a complete polarity sublattice of $\mathbf{3}^U$ satisfying (C1)--(C3), then by 
Theorem~\ref{Thm:main}(a), $\mathcal{A(F)} \subseteq \mathcal{RS}$ and Theorem~\ref{Thm:main}(b) yields 
$\mathcal{RS} \subseteq \mathcal{A(F)}$. Therefore, $\mathcal{RS} = \mathcal{A(F)}$.  \qed
\end{proof}
  
\begin{example} Let $U = \{a,b,c\}$.
  
(a) Suppose that $\mathcal{F} \subseteq \mathbf{3}^U$ consists of the following maps:
\[
\begin{array}{llll}
f_1 \colon  a \mapsto 0, b \mapsto 0, c \mapsto 0; & f_2 \colon  a \mapsto u, b \mapsto u, c \mapsto 0; &
f_3 \colon  a \mapsto 0, b \mapsto 0, c \mapsto u; & f_4 \colon a \mapsto u, b \mapsto u, c \mapsto u; \\[1mm]
f_5 \colon  a \mapsto 1, b \mapsto 1, c \mapsto u; & f_6 \colon  a \mapsto u, b \mapsto u, c \mapsto 1; &
f_7 \colon a \mapsto 1, b \mapsto 1, c \mapsto 1.  & 
\end{array}
\]
Obviously, $\mathcal{F}$ is and its Hasse diagram is given in Figure~\ref{Fig:Kleene}(a).
Figure~\ref{Fig:Kleene}(b) contains the Hasse diagram of the corresponding approximations $\mathcal{A(F)}$. Note that elements of sets
are denoted simply by sequences of their elements. For instance, $\{a,b\}$ is denoted $ab$.
\begin{figure}
\centering
\includegraphics[width=110mm]{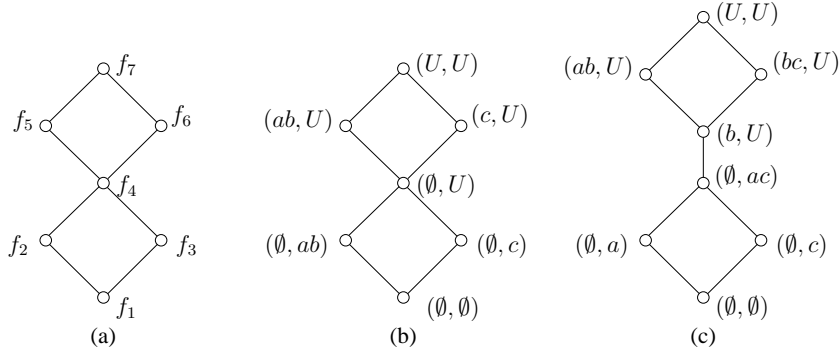}
\caption{\label{Fig:Kleene} A complete polarity sublattice of $\mathbf{3}^U$ not corresponding to any rough set algebra
is in (a), and (b) depicts its approximation pairs. The Hasse diagram of $\mathcal{RS} = \{ (X^\DOWN,X^\UP) \mid X \subseteq U \}$
is given in (c).}
\end{figure}

Now $C(\mathcal{F}) = \{\emptyset, \{a,b\}, \{c\},U\}$ and the corresponding quasiorder $\leq_\mathcal{F}$ is the equivalence whose equivalence
classes are $\{a,b\}$ and $\{c\}$. It is obvious that $\mathcal{A(F)}$ cannot be equal with rough set system $\mathcal{RS}$ induced by
$\leq_\mathcal{F}$, because $\mathcal{RS}$ is isomorphic to the product $\mathbf{2} \times \mathbf{3}$ and $\mathcal{A(F)}$ is not.
Let us now verify that conditions (C1)--(C3) do not hold.

The element $c$ is an $\mathcal{F}$-singleton. Now $c \in S(f_3)$, but $c \notin C(f_3)$. Therefore, (C1) does not hold.

Let us first compute the map $f_x = \bigwedge \{ f \mid \mathcal{F} \mid f(x) = 1\}$ for each $x \in U$:
\[ f_a = f_b = f_5 \wedge f_7 = f_5 \quad \mbox{and} \quad f_c = f_6 \wedge f_7 = f_6.\]
Now, for example, $C(f_b) = C(f_5) = \{a,b\} \nsubseteq \{b\}$, meaning that (C2) is not true.

The equivalence $\Theta$ has four classes:
\[ \{ f_1,f_2,f_3,f_4\}, \{f_5\}, \{f_6\}, \{f_7\}. \]
Now we have $C(f_6) = \{c\} = S(f_3)$, but $S(f_6) = U \nsubseteq \{c\} = S(f_3)$. Because $f_6$ is the only element in its $\Theta$-class, this
means that (C3) does not hold.

(b) Let us consider a quasiorder $\leq$ on $U$ such that
\[ [a) = \{a,b\}, \ [b) = \{b\}, \ [c) = \{b,c\} .\]

The Hasse diagram of $\mathcal{RS} = \{ (X^\DOWN,X^\UP) \mid X \subseteq U \}$ is depicted in Figure~\ref{Fig:Kleene}(c). Using
(\ref{Eq:characteristic}), we form the corresponding functions $U \to \mathbf{3}$:
\[
\begin{array}{lcl}
  f_{(\emptyset,\emptyset)} \colon  a \mapsto 0, b \mapsto 0, c \mapsto 0; & \quad & f_{(\emptyset,\{a\})} \colon  a \mapsto u, b \mapsto 0, c \mapsto 0;\\[1mm]
  f_{(\emptyset,\{c\})} \colon a \mapsto 0, b \mapsto 0, c \mapsto u; & \quad & f_{(\emptyset,\{a,c\})} \colon   a \mapsto u, b \mapsto 0, c \mapsto u;\\[1mm]
  f_{(\{b\},U)} \colon  a \mapsto u, b \mapsto 1, c \mapsto u; & \quad & f_{(\{a,b\},U)} \colon  a \mapsto 1, b \mapsto 1, c \mapsto u;\\[1mm]
  f_{(\{b,c\},U)} \colon  a \mapsto u, b \mapsto 1, c \mapsto 1; & \quad & f_{(U,U)} \colon  a \mapsto 1, b \mapsto 1, c \mapsto 1.
\end{array}
\]
Condition (C1) has now the interpretation that if an $\mathcal{F}$-singleton belongs to an upper approximation $X^\UP$ of some subset $X$ of $U$,
it belongs also to the corresponding lower approximation $X^\DOWN$. By the proof of Proposition~\ref{prop:necessary}, $C(f_x)$ corresponds to
the lower approximation $\{x\}^\DOWN$, which is always included in $\{x\}$. This is expressed in (C2). Conditions (C1) and (C2) hold
actually for all reflexive binary relations.

In terms of rough sets, condition (C3) can be written as: If $X^\DOWN \subseteq Y^\UP$ and $\mathcal{H} = \{ Z \subseteq U \mid Z^\DOWN = X^\DOWN\}$,
then $\bigcap \{ Z^\UP \mid  Z \in \mathcal{H} \} \subseteq Y^\UP$.
This condition does not hold for tolerances, for instance. Let $R$ be a tolerance on $U$ such that $R(a) = \{a,b\}$, $R(b) = U$ and $R(c) = \{b,c\}$.
Let $X = \{a,b\}$ and $Y = \{a\}$. Now $X^\DOWN = \{a\} \subseteq \{a,b\} = Y^\UP$.
It can be easily checked that $\mathcal{H} = \{X\}$. Now $X^\UP = U \nsubseteq \{a,b\} = Y^\UP$.
\end{example}

We end this work by the following theorem answering to Question~\ref{q:equivalence}.

\begin{theorem} \label{thm:equivalence}
If $\mathcal{F} \subseteq \mathbf{3}^U$, then $\mathcal{A(F)} = \mathcal{RS}$ for some equivalence on $U$ if and
only if $\mathcal{F}$ is a complete {\L}ukasiewicz subalgebra of $\mathbf{3}^U$ satisfying (C1)--(C3).
\end{theorem}

\begin{proof} Assume that $\mathcal{A(F)} = \mathcal{RS}$ for some equivalence on $U$. Then, by Theorem~\ref{thm:quasiorder},
$\mathcal{F}$ is a complete polarity sublattice of $\mathbf{3}^U$ satisfying (C1)--(C3). 
By Proposition~\ref{Prop:ImplyLuka} it is enough to show that $\mathcal{F}$ is closed with respect to $^*$. By Lemma~\ref{lem:support_star}, 
\[ \mathcal{A}(f^*) = (C(f^*), S(f^*)) = (S(f)^c, S(f)^c) = (C(f),(S(f))^* \in \mathcal{RS} = \mathcal{A(F)} .\]  
We have that $f^* = \varphi^{-1}(\mathcal{A}(f^*))$ belongs to $\mathcal{F}$. Thus,
$\mathcal{F}$ is a complete {\L}ukasiewicz subalgebra of $\mathbf{3}^U$.

Conversely, suppose that $\mathcal{F}$ is a complete {\L}ukasiewicz subalgebra of $\mathbf{3}^U$ satisfying (C1)--(C3).
By Theorem~\ref{thm:quasiorder}, $\mathcal{A(F)} = \mathcal{RS}$ for some quasiorder $R$. We proved in \cite[Prop.~4.5]{JarRad11}
that for any quasiorder $R$, $\mathcal{RS}$ forms a three-valued {\L}ukasiewicz algebra if and only if $R$ is an equivalence.
This completes the proof. \qed
\end{proof}

\section*{Some concluding remarks}

In this work we have answered the question what conditions a collection $\mathcal{F}$ of 3-valued functions on $U$ must fulfill
so that there exists a quasiorder $\leq$ on $U$ such that the set $\mathcal{RS}$ of rough sets defined by $\leq$ coincides with the
set $\mathcal{A(F)}$ of approximation pairs defined by $\mathcal{F}$. Furthermore, we give a new representation of rough sets determined
by equivalences in terms of three-valued {\L}ukasiewicz algebras of three-valued functions.

It is known that for tolerances determined by irredundant coverings on $U$, the induced rough set structure $\mathcal{RS}$ is a
regular pseudocomplemented Kleene algebra, but now $\mathcal{RS}$ is not a complete sublattice of the product $\wp(U) \times \wp(U)$;
see \cite{JarRad14, JarRad17, JarRad19}. This means that if $\mathcal{F}$ is a collection of three-valued maps such that
$\mathcal{A(F)} = \mathcal{RS}$, then obviously $\mathcal{F}$ is not a complete sublattice of $\mathbf{3}^U$. A natural question then
is what properties $\mathcal{F}$ needs to have to define a rough set system determined by a tolerance induced by an irredundant covering.

Finally, in this work we have considered approximation pairs defined by three-valued functions. But one could change $\mathbf{3}$ to some
other structure. For instance, $\mathbf{3}$ could be replaced by 4-element lattice introduced in \cite{belnap1977}, where
$\mathbf{L4}$ denotes the lattice $\mathbf{F} < \mathbf{Both}, \mathbf{None} < \mathbf{T}$, where $\mathbf{F}$ means `false', $\mathbf{Both}$ means
`both true and false', $\mathbf{None}$ means `neither true nor false', and $\mathbf{T}$ means `true'. In such a setting we could
consider, for instance, the approximation pairs formed of level set of functions $f \colon U \to \mathbf{L4}$, that is,
$f_\alpha = \{ x \in U \mid f(x) \geq \alpha\}$, where $\alpha$ belongs to $ \mathbf{L4}$.

\section*{Declaration of competing interest and authorship conformation}

The authors declare that they have no known competing financial interests or personal relationships that could have appeared to
influence the work reported in this paper. Both authors have participated in drafting the article and approve the final version.

\section*{Acknowledgement}

The authors would like to thank the anonymous referees for the significant time and effort
they put in to provide expert views on our original manuscript. Also a question posed by the Area Editor
urged us to add more motivations from the approximate reasoning standpoint.

The research of the second author started as part of the TAMOP-4.2.1.B-10/2/KONV-2010-0001
project, supported by the European Union, co-financed by the European Social Fund 113/173/0-2.


\end{document}